\documentclass{article}
\usepackage{amsmath,amsthm,amssymb,amscd}
\usepackage[all,cmtip,color]{xy}
\usepackage[german,frenchb,english]{babel}
\usepackage{amsmath}
\usepackage{amssymb}
\usepackage{hyperref}
\usepackage{mathrsfs} 
\usepackage{amssymb}
\usepackage{todonotes}
\usepackage{amssymb}
\usepackage[all,cmtip]{xy}
\input diagxy
\usepackage{hyperref}
\usepackage[margin = 2.5 cm]{geometry}

\usepackage[export]{adjustbox}
\usepackage{graphicx}

\usepackage{xcolor}

\DeclareMathOperator{\Hom}{\mathsf{Hom}}

\DeclareMathOperator{\fib}{\mathsf{fib}}

\DeclareMathOperator{\Nb}{\mathbb{N}}

\DeclareMathOperator{\Sh}{Sh}

\DeclareMathOperator{\Grpd}{\infty-\mathsf{Gpd}}

\DeclareMathOperator{\Fb}{\mathbb{F}}

\newcommand{\Cc}{\mathcal{C}}

\newcommand{\Ec}{\mathcal{E}}

\DeclareMathOperator{\colim}{\mathsf{colim}}
\renewcommand{\lim}{\mathsf{lim}}

\DeclareMathOperator{\Sets}{\mathsf{Set}}
\DeclareMathOperator{\Rings}{\mathsf{Rng}}

\DeclareMathOperator{\Perf}{\mathsf{Perf}}

\DeclareMathOperator{\Pb}{\mathbb{P}}

\newcommand{\C}{\mathsf{C}}

\newcommand{\Spectra}{Sp}

\DeclareMathOperator{\rk}{\mathsf{rk}}
\DeclareMathOperator{\id}{id}

\DeclareMathOperator{\FMod}{\mathsf{FMod}}

\DeclareMathOperator{\Gpd}{\mathsf{Grpd}}

\DeclareMathOperator{\Loc}{\mathsf{Loc}}

\DeclareMathOperator{\op}{\mathsf{op}}

\DeclareMathOperator{\Ab}{\mathbb{A}}

\DeclareMathOperator{\Coh}{\mathsf{Coh}}

\DeclareMathOperator{\Kb}{\mathbb{K}}

\DeclareMathOperator{\Map}{\mathsf{Map}}
\DeclareMathOperator{\G}{\mathbb{G}}

\DeclareMathOperator{\Ee}{\mathcal{E}}

\DeclareMathOperator{\Sch}{\mathsf{Sch}}

\DeclareMathOperator{\Pc}{\mathcal{P}}

\DeclareMathOperator{\Spec}{Spec}

\DeclareMathOperator{\Sym}{Sym}

\DeclareMathOperator{\Oo}{\mathcal{O}}
\DeclareMathOperator{\red}{red}

\DeclareMathOperator{\Pic}{\mathsf{Pic}}

\DeclareMathOperator{\Zb}{\mathbb{Z}}
\DeclareMathOperator{\coh}{coh}
\DeclareMathOperator{\gr}{gr}

\DeclareMathOperator{\grdet}{\mathsf{det}^{\Zb}}
\DeclareMathOperator{\grPic}{\mathsf{Pic}^{\Zb}}

\DeclareMathOperator{\DR}{\mathsf{DR}}

\DeclareMathOperator{\Ac}{\mathcal{A}}

\DeclareMathOperator{\eqgr}{\widetilde{\gr}}

\DeclareMathOperator{\cri}{\mathsf{cri}}

\DeclareMathOperator{\bfOmega}{\mathbf{\Omega}}

\DeclareMathOperator{\Adm}{\mathsf{Adm}}

\begin{document}
\newtheorem{definition}{Definition}[section]
\newtheorem{theorem}[definition]{Theorem}
\newtheorem{proposition}[definition]{Proposition}
\newtheorem{corollary}[definition]{Corollary}
\newtheorem{conj}[definition]{Conjecture}
\newtheorem{lemma}[definition]{Lemma}
\newtheorem{rmk}[definition]{Remark}
\newtheorem{cl}[definition]{Claim}
\newtheorem{example}[definition]{Example} 
\newtheorem{claim}[definition]{Claim}
\newtheorem{ass}[definition]{Assumption}
\newtheorem{warning}[definition]{Dangerous Bend}
\newtheorem{porism}[definition]{Porism}
\newtheorem{situation}[definition]{Situation}
\newtheorem{lemma-definition}[definition]{Lemma-Definition}
\newtheorem{construction}[definition]{Construction}
\newtheorem{convention}[definition]{Convention}
\newtheorem{question}[definition]{Question}

\author{Michael Groechenig}
\date{}

\title{Higher de Rham epsilon factors\let\thefootnote\relax\footnotetext{This project has received funding from the European Union's Horizon 2020 research and innovation programme under the Marie Sk\l odowska-Curie Grant Agreement No. 701679. \\ \includegraphics[height = 1cm,right]{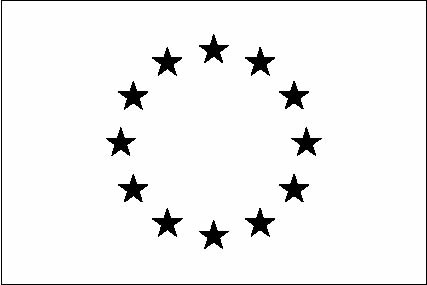}
} 
}
\maketitle 

\abstract{This article is devoted to the study of a higher-dimensional generalisation of de Rham epsilon lines. To a holonomic $D$-module $M$ on a smooth variety $X$ and a generic tuple of $1$-form $(\nu_1,\dots,\nu_n)$, we associate a point of the $K$-theory space $K(X,Z)$. If $X$ is proper this $K$-theory class is related to the de Rham cohomology $R\Gamma_{dR}(X,M)$.  The novel feature of our construction is that $Z$ is allowed to be of dimension $0$. Furthermore, we allow the tuple of $1$-forms to vary in families, and observe that this leads naturally to a crystal akin to the epsilon connection for curves. Our approach is based on combining a construction of Patel with a homotopy invariance property of algebraic $K$-theory with respect to $(\Pb^1,\infty)$. This homotopical viewpoint leads us naturally to the definition of an epsilon connection in higher dimensions. Along the way we prove the compatibility of Patel's epsilon factors with the graded lines defined by Deligne and Beilinson--Bloch--Esnault in the case of curves.}

\tableofcontents
 
\section{Introduction}

Let $\mathcal{F}$ be a constructible $\ell$-adic sheaf on a smooth proper curve $X/\Fb_q$. 
%The determinant of the cohomology of $\mathcal{F}$ is the $1$-dimensional $\ell$-adic vector space 
%$$\det(X,\mathcal{F}) = \bigotimes_{i=0}^2 \left({\bigwedge}^{top} H^i(X_{\bar{\Fb}_q},\mathcal{F})\right)^{(-1)^i}.$$
%This $\ell$-adic line is endowed with an action of the Frobenius, and the corresponding eigenvalue $\varepsilon(\mathcal{F})$ is an important arithmetic invariant:
Recall that to $\mathcal{F}$ one associates an $L$-function $L(\mathcal{F},t)$ which is a rational function in $t$ (with $\ell$-adic coefficients). Poincar\'e duality implies the functional equation
$L(\mathcal{F},\frac{1}{qt}) = \varepsilon(\mathcal{F})\cdot L(\mathcal{F}^{\vee},t).$
The constant $\varepsilon(\mathcal{F})$ arising in this functional equation is described in terms of the Frobenius eigenvalue on the determinant of cohomology
$\varepsilon(\mathcal{F}) = \mathsf{tr}(-F,\det(H^*(X,\mathcal{F}))).$
We also refer to it as the \emph{global epsilon factor}. 

A formalism of \emph{local epsilon factors} allows one to associate to a non-zero rational $1$-form $\nu$ on $X$ and a closed point $x \in X$ an (invertible) $\ell$-adic number $\varepsilon_{\nu,x}(\mathcal{F})$ which is equal to $1$ for almost all closed points $x$, such that we have a \emph{product formula}
\begin{equation}\label{productformula}\varepsilon(\mathcal{F}) = \prod_{x \in X_{cl}} \varepsilon_{\nu,x}(\mathcal{F}).\end{equation}
Furthermore, the quantity $ \varepsilon_{\nu,x}(\mathcal{F})$ is only to depend on the constructible sheaf $\mathcal{F}$ (and the form $\nu$) near $x$, that is, the restriction to the formal completion of $X$ at $x$.

The existence of a formalism of local epsilon factors was proven by Laumon \cite{MR908218} using global methods. A purely local definition of $\varepsilon_{\nu,x}(\mathcal{F})$ remains elusive and would be desirable.
%\subsection{The de Rham analogue}
Replacing finite fields by the complex numbers $\mathbb{C}$ (or a field of characteristic zero) and $\ell$-adic sheaves by holonomic $D$-modules, the determinant line of de Rham cohomology remains an interesting invariant (e.g. due to its connection to period matrices). It's been observed by Deligne in his farewell seminar at IHES, and Beilinson--Bloch--Esnault \cite{MR1988970} that one can define graded lines (that is, a $1$-dimensional graded vector spaces $\Ee_{\nu,x}(\mathcal{F})$ for a holonomic $D$-module $\mathcal{F}$ and a non-zero rational $1$-form on $X$, such that we have the following analogue of formula \eqref{productformula}:
\begin{equation}\label{eqn:BBE}\det(X,\mathcal{F}) = \bigotimes_{i=0}^2 \left({\bigwedge}^{top} H^i(X,\mathcal{F})\right)^{(-1)^i} \simeq \bigotimes_{x \in X_{cl}} \Ee_{\nu,x}(\mathcal{F}).\end{equation}
We recall this construction in \ref{BBE}.
Their construction for de Rham cohomology has a Betti counterpart, and hence lends itself to define \emph{epsilon periods} (see \cite{beilinson2005periods} for a construction of epsilon periods using Fourier transform, and Beilinson's \cite{MR2330165, MR2562452}).

A generalisation of de Rham epsilon lines to higher dimensions was developed by Patel in \cite{MR2981817}, see also \ref{Patel} for a summary of his theory. Let $X$ be a smooth variety over a field $k$ of characteristic $0$, $U \subset X$ an open subset, and $\nu \in \Omega_X^1(U)$. We denote by $K(D_X,{\nu})$ the $K$-theory spectrum of locally finitely presented $D$-modules $M$ on $X$, such that $\nu(U)$ does not intersect the singular support of $M$. And by $K(X,D)$ the $K$-theory spectrum of coherent sheaves on $X$ which are set-theoretically supported on $D = X \setminus U$.

\begin{theorem}[Patel]
There exists a morphism of spectra $\Ee^{P}_{\nu}\colon \Kb(D_X,\nu) \to \Kb(X,D)$,  such that for $X$ proper, we have a commutative diagram
\[
\xymatrix{
K(D_X,{\nu}) \ar[r]^{\Ec_{\nu}^P} \ar[rd]_{R\Gamma_{dR}} & K(X,D) \ar[d]^{R\Gamma} \\
& K(k).
}
\]
\end{theorem}

Patel's epsilon factors take values in the $K$-theory spectrum $K(X,D)$. If $D$ is proper, we can apply $R\Gamma(X,-)$ and obtain a perfect complex of vector spaces. The graded determinant thereof will be denoted by 
$\varepsilon^P_{\nu}(M)$
and referred to as Patel's epsilon line. While this is a pleasant generalisation of the theory for curves to higher dimensions, it would be desirable to find the higher-dimensional counterparts of the following features of Beilinson--Bloch--Esnault's theory in dimension $1$:
\item[(a)] Is there a refinement of Patel's epsilon factors which decomposes as a product $\bigotimes_{x \in X_{\rm closed}} \varepsilon_{\nu,x}(M)$ similar to equation \eqref{eqn:BBE}?
\item[(b)] Can we vary $\nu$ in families with respect to a commutative $k$-algebra $A$, and is there an analogue of the epsilon connection of \cite{MR1988970}?
\item[(c)] Is Patel's epsilon factor isomorphic to the one defined for curves by Deligne and Beilinson--Bloch--Esnault?

In this paper we give an affirmative answer to these questions. We fix a commutative $k$-algebra $A$ (where $k$ denotes again a field of characteristic $0$). Let $Z \subset X$ be a proper subset (possibly $0$-dimensional), and $X \setminus Z = \bigcup_{i = 1}^m U_i$ an open covering. Furthermore, we choose for every $i \in \{1,\dots, m\}$ a section $\nu_i$ of $\Omega^1_{(U_i)_A/A}$. We refer to the tuple $(\nu_1,\dots,\nu_m)$ by $\underline{\nu}$, and denote by $K(D_X,\underline{\nu})$ the $K$-theory of the full subcategory of holonomic $D$-modules $M$, such that for every $\emptyset \neq I \subset \{1,\dots,m\}$ the singular support of $M$ does not intersect the affine space $\sum_{i\in I}\lambda_i \nu_i$ with $\sum_{i\in I} \lambda_i = 1$. The following result is proven at the end of Subsection \ref{sub:relative}.

\begin{theorem}\label{thm:mainmain}
There exists a morphism of spectra $\Ee_{\underline{\nu},A}\colon \Kb(D_X,\underline{\nu}) \to \Kb(X_A,Z_A)$,  such that for $X$ proper, we have a commutative diagram
\[
\xymatrix{
K(D_X,{\nu}) \ar[r]^{\Ec_{\underline{\nu},A}} \ar[d]_{R\Gamma_{dR}} & K(X_A,Z_A) \ar[d]^{R\Gamma} \\
K(k) \ar[r]^{\otimes_k A} & K(A).
}
\]
\end{theorem}

Taking the graded determinant of $R\Gamma \circ \Ee_{\underline{\nu}}$ one obtains a graded line which decomposes naturally into factors 
$$\varepsilon_{\nu}\simeq \bigotimes_{x \in \pi_0(Z)} \varepsilon_{\underline{\nu},x}.$$
If $Z$ is zero-dimensional, this yields the promised factorisation format in higher dimensions. The epsilon connection is introduced in Section \ref{epsilon}. The following is proven in Section \ref{epsilon}, and responds to question (b) above.

\begin{theorem}
The local epsilon line $\varepsilon_{\underline{\nu}}(M) / \Spec A$ of a holonomic $D$-module carries a natural crystal structure. Furthermore, if $X$ is smooth and proper, $\varepsilon_{\underline{\nu}}(M)$ is isomorphic to the constant crystal induced by $\det(R\Gamma_{dR}(X,M)) \otimes_k A$.
\end{theorem}

The solution to (a) and (b) is based on the same principle: the invariance property of algebraic $K$-theory along Zariski-locally trivial $\Pb^N$-fibrations with a modulus condition at $\infty$ (in the case of regular base schemes this is simply $\Ab^1$-invariance of $K$-theory).
The comparison between Patel's and Beilinson--Bloch--Esnault's theory is the content of Theorem \ref{thm:comparison}.

\medskip

\noindent\textit{Categorical conventions.} We work in the framework of $\infty$-categories (mostly stable $\infty$-categories) as discussed in Lurie's \cite{Lurie:ha}. It should be possible for readers without prior exposure to the theory of $\infty$-categories to follow the outline of our construction.

\medskip

\noindent\textit{Acknowledgements.}  The author thanks Oliver Braunling, H\'el\`ene Esnault, Javier Fres\'an, Deepam Patel, Simon P\'epin-Lehalleur, Kay R\"ulling, Markus Roeser and Jesse Wolfson for interesting conversations about epsilon factors, differential equations and algebraic $K$-theory, and for their comments on a preliminary version of this article.

\section{De Rham epsilon factors: what we know so far}

\subsection{Beilinson--Bloch--Esnault's epsilon lines}\label{BBE}

The theory of de Rham epsilon factors for irregular flat connections on curves was developed (independently) by Deligne (in his farewell seminar at IHES) and Beilinson--Bloch--Esnault in \cite{MR1988970}. 

\begin{situation}\label{situation:BBE}
Let $k$ be a field of characteristic $0$ and $X$ a smooth $k$-scheme of dimension $1$. Furthermore, we assume either that $X$ is a smooth curve (possibly affine) or that $X$ is a \emph{trait}, that is, isomorphic to $\Spec k'[[t]]$ for some finite field extension $k'/k$. We choose an open subset $j\colon U \hookrightarrow X$ with closed complement $D \subset X$. We let $M$ be a holonomic $D$-module on $X$, such that $M|_U$ is $\Oo$-coherent, that is, corresponds to a vector bundle with a flat connection $\Ec=(E,\nabla)$ on $U$. We denote by $\nu$ a nowhere vanishing $1$-form on $U$.
\end{situation}

A \emph{graded line} over a commutative ring $R$ is a pair $(L,n)$ where $L$ is an invertible $R$-module and $n \in \Zb$ an integer. The groupoid of graded lines will be denoted by $\grPic(R)$. Addition in $\Zb$ and the tensor product $\otimes$ of invertible sheaves leads to a monoidal structure $\otimes$ on $\grPic(R)$. We endow $\grPic$ with the symmetric monoidal structure given by the symmetry constraint
\[
\xymatrix{
(L_1,n_1) \otimes (L_2,n_2) \ar[r] \ar[d] & (L_2,n_2) \otimes (L_1,n_1) \ar[d] \\
(L_1 \otimes L_2,n_1 + n_2) \ar[r]^{(-1)^{n_1n_2}c} & (L_2 \otimes L_1, n_1 + n_2),
}
\]
where $c\colon L_1 \otimes L_2 \to L_2 \otimes L_1$ is the canonical isomorphism.

Under the circumstances described in Situation \ref{situation:BBE}, the theory of \emph{loc. cit.} can be used to define a graded line $\varepsilon_{\nu}(M)$ which only depends on the geometry of $M$ and $X$ infinitely near $D$. We prepare the ground by introducing relative determinant lines. This has the desirable effect that we obtain a \emph{symmetric monoidal} functor from the groupoid of finitely generated projective $R$-modules (discarding all non-invertible maps) to $\grPic(R)$, denoted by
$$\grdet\colon (P^f(R))^{\times} \to \grPic(R)$$
which sends $V$ to $(\det V, \rk V)$.

\begin{lemma-definition}\label{lemma-defi:relative}
Let $\Ec$ be a vector bundle on $U$. 
\begin{enumerate}
\item[(a)] A lattice in $\Ec$ is an $\Oo$-coherent subsheaf $L \subset j_* \Ec$, such that $L|_U = \Ec$.

\item[(b)] For a pair of lattices $L_1, L_2 \subset j_*\Ec$ we define a graded $k$-line 
$$\grdet(L_1:L_2) = \grdet(\Gamma(L_1/L)) \otimes (\grdet(\Gamma(L_2/L)))^{-1},$$
where $L$ is an arbitrarily chosen lattice satisfying $L \subset L_1 \cap L_2$. We claim that this graded line only depends on $(L_1,L_2)$ up to a unique isomorphism.
\item[(c)] The relative graded determinant $\grdet(L_1:L_2)$ factorises as a tensor product (with almost all but finitely many factors trivialised)
$$\grdet(L_1:L_2) \simeq \bigotimes_{x \in X_{\rm closed}} \grdet(L_1:L_2)_x.$$
\end{enumerate}
\end{lemma-definition}

\begin{proof}
It suffices to show that $\grdet(L_1:L_2)$ is well-defined, that is, independent of the choice of $L$. For $L,L' \subset L_1 \cap L_2$ we may choose $L'' \subset L \cap L'$. We have $\grdet(\Gamma(L_1/L'') \simeq \grdet(\Gamma(L_1/L) \otimes \grdet(\Gamma(L/L'')$, and similarly $\grdet(\Gamma(L_2/L'') \simeq \grdet(\Gamma(L_2/L) \otimes \grdet(\Gamma(L/L'')$. This induces an isomorphism
$$\grdet(L_1:L_2) = \grdet(\Gamma(L_1/L)) \otimes (\grdet(\Gamma(L_2/L)))^{-1}\simeq \grdet(\Gamma(L_1/L'')) \otimes (\grdet(\Gamma(L_2/L'')))^{-1}.$$
Reversing the role of $L$ and $L'$ we obtain an isomorphism
$$\grdet(L_1:L_2) = \grdet(\Gamma(L_1/L)) \otimes (\grdet(\Gamma(L_2/L)))^{-1}\simeq \grdet(\Gamma(L_1/L')) \otimes (\grdet(\Gamma(L_2/L')))^{-1}.$$
We leave the verification of the fact that this isomorphism is independent of the choice of $L''$ to the reader.

The factorisation property (c) follows by defining $\grdet(L_1:L_2)_x = \grdet(\Gamma((L_1/L)_x) \otimes \grdet(\Gamma((L_2/L)_x))^{-1}$, where $(L_i/L)_x$ denotes the stalk at a closed point $x \in X$. Since $L_i/L$ is a skypscraper sheaf supported at closed points, we have 
$$\Gamma(L_i/L) = \bigoplus_{x \in X_{\rm closed}} (L_i/L)_x.$$
This concludes the proof.
\end{proof}

The multiplicativity of graded determinants with respect to short exact sequences yields the following transitivity property.

\begin{lemma}\label{lemma:transitivity}
For every triple of lattices $L_1,L_2,L_3 \subset j_*\Ec$ we have an isomorphism
$t_{123}\colon \grdet(L_1:L_2) \otimes \grdet(L_2:L_3) \to^{\simeq} \grdet(L_1:L_3)$, such that for every quadruple $L_1,L_2,L_3,L_4 \subset j_*\Ec$ we get a commutative diagram
\[
\xymatrix{
\grdet(L_1:L_2) \otimes \grdet(L_2:L_3) \otimes \grdet(L_3:L_4) \ar[r]^-{t_{123} \otimes \id} \ar[d]_{\id \otimes t_{234}} & \grdet(L_1:L_3) \otimes \grdet(L_3:L_4) \ar[d]^{t_{134}} \\
\grdet(L_1:L_2) \otimes \grdet(L_2:L_4) \ar[r]^-{t_{124}} & \grdet(L_1:L_4)
}
\]
of graded lines.
Similarly for $\grdet(L_i:L_j)_x$.
\end{lemma}

\begin{lemma-definition}\label{lemma-defi:BBE} 
\begin{enumerate}
\item[(a)] A \emph{pair of lattices} for $\Ec=(E,\nabla)$ are locally free $\Oo$-coherent subsheaves $L, N\hookrightarrow j_*E$, such that $L|_U = E$, $N|_U = E$ and satisfying $\nabla(L) \subset N \otimes \Omega_X^1(D)$.

\item[(b)] We denote by $\DR(L,N))$ the complex of sheaves
$[L \to^{\nabla} N \otimes \Omega_X^1(D)].$

\item[(c)] We define 
$$\tilde{\varepsilon}_{\nu}(X,\Ec) = \grdet(L:\nu^{-1}(N \otimes \Omega_X^1)) \otimes \left(\grdet(R\Gamma(\DR(L,N))\right)^{-1} \otimes \grdet(R\Gamma_{dR}(U,\Ec).$$

\item[(c)] The de Rham epsilon line is defined to be the graded line $\varepsilon_{\nu}(X,M) = \tilde{\varepsilon}_{\nu}(E,\nabla) \otimes \grdet(R\Gamma_{dR,D}(M)).$
\end{enumerate}
\end{lemma-definition}

\begin{proof}
Only point (c) requires clarification. We need to verify that the definition given there is independent of choices. Let $(L_1,N_1)$ be two pairs of lattices for $(E,\nabla)$. Without loss of generality we may assume $L_1 \subset L_2$ and $N_1 \subset N_2$. The transitivity property for the relative determinant (Lemma \ref{lemma:transitivity}) yields
\begin{eqnarray*}
\grdet(L_2:\nu^{-1}(N_2 \otimes \Omega_X^1(D))) \simeq \\ \simeq  \grdet(L_2:L_1) \otimes \grdet(L_1:\nu^{-1}(N_1 \otimes \Omega_X^1(D))) \otimes \grdet(\nu^{-1}(N_1 \otimes \Omega_X^1(D)):\nu^{-1}(N_2 \otimes \Omega_X^1(D))).
\end{eqnarray*}
Similarly, multiplicativity of graded determinants with respect to short exact sequences implies
$$\grdet(R\Gamma(\DR(L_2,N_2))) \simeq \grdet(\Gamma(L_1/L_2)) \grdet(R\Gamma(\DR(L_1,N_1))) \otimes \grdet(\Gamma(N_2/N_1)).$$
By definition we have $\grdet(\Gamma(L_1/L_2)) = \grdet(L_1:L_2)$. Therefore we see that 
\begin{eqnarray*}
\grdet(L_1:\nu^{-1}(N_1 \otimes \Omega_X^1(D))) \otimes \left(\grdet(R\Gamma(\DR(L_1,N_1))\right)^{-1} \simeq \\ \simeq \grdet(L_2:\nu^{-1}(N_2 \otimes \Omega_X^1(D))) \otimes \left(\grdet(R\Gamma(\DR(L_2,N_2))\right)^{-1}.\end{eqnarray*}
This concludes the proof.
\end{proof}

If $(L,N)$ satisfies the property 
$$[L \to^{\nabla} N \otimes \Omega_X^1(D)] \simeq Rj^{dR}_*\Ec,$$
then one says that $(L,N)$ is a \emph{good lattice pair}. According to a theorem of Deligne, good lattice pairs always exist (\cite[p. 110-112]{MR0417174}).

\begin{corollary}\label{cor:good}
Let $(L,N)$ be a good lattice pair for $\Ec$. Then we have an isomorphism
$$\tilde{\varepsilon}_{\nu}(X,\Ec) \simeq \grdet(M:\nu^{-1}(N \otimes \Omega_X^1(S))).$$
\end{corollary}

\begin{rmk}
This is the definition the de Rham epsilon line given by Deligne in his farewell seminar at IHES. 
\end{rmk}

Using the factorisation property (c) of Lemma-Definition \ref{lemma-defi:relative}, one can define local epsilon factors.

\begin{lemma-definition}
There exists $\tilde{\varepsilon}^{BBE}_{\nu,x}(\Ec)$ and $\varepsilon^{BBE}_{\nu,x}(M)$, such that 
$$\tilde{\varepsilon}^{BBE}_{\nu}(X,\Ec)\simeq \bigotimes_{x \in U_{\rm closed}} \tilde{\varepsilon}_{\nu,x}(\Ec),$$
$$\varepsilon^{BBE}_{\nu}(X,M) \simeq  \bigotimes_{x \in U_{\rm closed}} {\varepsilon}_{\nu,x}(M).$$
\end{lemma-definition}

The product formula for a proper smooth varieties $X$ follows directly from the definition, using the defining property of good lattice pairs. We denote by $H^*_{dR}(X,M)$ the de Rham cohomology of a holonomic $D$-module on $X$ which lives in degrees $[0,2]$.

\begin{theorem}[Product formula]
Suppose that in Situation \ref{situation:BBE}, $X$ is a smooth proper curve. Let $M$ be a (bounded complex of) holonomic $D$-module on $X$. Then, we have an isomorphism of graded lines 
$$\grdet(H^*_{dR}(X,M)) \simeq \varepsilon^{BBE}_{\nu}(X,M) \simeq \bigotimes_{x \in X_{\rm closed}} \varepsilon^{BBE}_{\nu,x}(M).$$ 
\end{theorem}

\begin{proof}
The proof is reduced to the case $M = Rj_*^{dR}\Ec$ where $\Ec = (E,\nabla)$ is a flat vector bundle on an open subset $U \subset X$. We have to produce an isomorphism
$$\grdet(H^*_{dR}(U,\Ec)) \simeq \tilde{\varepsilon}^{BBE}_{\nu}(\Ec).$$
For a good lattice pair $(L,N)$ we have 
$$\grdet(R\Gamma(X,L)) \otimes \grdet(R\Gamma(X,N))^{-1} \simeq \grdet(R\Gamma(U,\Ec)).$$
The left hand side is isomorphic to $\grdet(L:\nu^{-1}(N \otimes \Omega_X^1(D))) \simeq \tilde{\varepsilon}^{BBE}_{\nu}(\Ec).$
\end{proof}

The transitivity property of relative determinants yields the following structure on de Rham epsilon lines.

\begin{lemma-definition}\label{lemma-defi:BBE_beta}
For every short exact sequence of holonomic $D$-modules 
$$\xi\colon M' \hookrightarrow M \twoheadrightarrow M''$$
we have an isomorphism of graded lines
$$\beta^{BBE}_{\xi}\colon \varepsilon^{BBE}_{\nu}(M) \simeq \varepsilon^{BBE}_{\nu}(M') \otimes \varepsilon^{BBE}_{\nu}(M''),$$
such that for every diagram
\[
\xymatrix{
M \ar@{^(->}[r] & N \ar@{->>}[d] \ar@{^(->}[r] & P \ar@{->>}[d] \\
& N/M \ar@{^(->}[r]  & P/M \ar@{->>}[d] \\
& & P/N
}
\]
we have a commutative diagram of isomorphisms of graded lines
\[
\xymatrix{
\varepsilon^{BBE}_{\nu}(P) \ar[r] \ar[d] & \varepsilon^{BBE}_{\nu}(N) \otimes \varepsilon^{BBE}_{\nu}(P/N) \ar[d] \\
\varepsilon^{BBE}_{\nu}(M) \otimes \varepsilon^{BBE}_{\nu}(P/M) \ar[r] & \varepsilon^{BBE}_{\nu}(M) \otimes \varepsilon^{BBE}_{\nu}(N/M) \otimes \varepsilon^{BBE}_{\nu}(N/M)^{-1} \otimes \varepsilon^{BBE}_{\nu}(P/M). 
}
\]
The same thing holds for the local epsilon factors $\varepsilon^{BBE}_{\nu,x}$.
Furthermore, if $X$ is a proper smooth curve, then $\beta^{\xi}$ is compatible with the isomorphism
$$\grdet(R\Gamma_{dR}(N)) \simeq \grdet(R\Gamma_{dR}(M)) \otimes \grdet(R\Gamma_{dR}(P)).$$
\end{lemma-definition}

\subsection{Patel's epsilon factor}\label{Patel}

In \cite{MR2981817} Patel introduced a formalism of de Rham epsilon factors for higher-dimensional schemes. This subsection is devoted to reviewing Patel's construction. We denote by $k$ a field of characteristic $0$.

\begin{rmk}
Henceforth we adopt the grading conventions of Patel's paper \cite{MR2981817} where de Rham cohomology $R\Gamma_{dR}$ is supported in the degrees $[-n,n]$, $n = \dim X$. Furthermore we emphasise that $D$-modules will be \emph{right} $D$-modules.
\end{rmk}

This convention will lead to the appearance of a shift when comparing Patel's epsilon factors with those defined by Deligne and Beilinson--Bloch--Esnault.

\begin{definition}\label{defi:infinitelynear}
Let $X$ be a scheme of finite presentation and $D \subset X$ a closed subset. We say that a finite presentation morphism of schemes $f\colon Y \to X$ is \emph{an isomorphism infinitely near $D$}, if the induced morphism of formal schemes $\widehat{Y}^{\mathsf{form}}_D \to \widehat{X}^{\mathsf{form}}_D$ is an isomorphism.
\end{definition}

\begin{situation}\label{situation:patel}
Let $X$ be a smooth separated $k$-scheme, $D \subset X$ a closed subset, and $\nu \in \Omega^1_X(X \setminus D)$ a $1$-form. We denote by $\Perf(D_X,S)$ the bounded derived $\infty$-category of perfect $D$-modules on $X$, and by $\Perf^{\nu}(D_X)$ the full subcategory of objects $M$ whose singular support $S$ does not intersect the graph of $\nu$.
\end{situation}

\begin{theorem}[Patel]\label{thm:patel}
There exists a morphism of spectra $\Ee^{P}_{\nu}\colon \Kb(D_X,\nu) \to \Kb(X,D)$, satisfying the following properties.
\begin{itemize}
\item[(a)] ``Excision": For a morphism of smooth varieties $f\colon Y \to X$ which is an isomorphism infinitely near $D \subset X$ we have a commutative diagram of spectra
\[
\xymatrix{
\Kb(D_X,{\nu}) \ar[r] \ar[d]_{\Ec_{\nu}^P} & \Kb(D_Y,{f^*\nu}) \ar[d]^{\Ec^P_{f^*\nu}} \\
\Kb(X,D) \ar[r] & \Kb(Y,f^{-1}(D)).
}
\]

\item[(b)] ``Product formula": If $X$ is proper (and $\dim X = n$) we have a commutative diagram
\[
\xymatrix{
\Kb(D_X,{\nu}) \ar[r]^{\Ec_{\nu}^P} \ar[rd]_{R\Gamma_{dR}} & \Kb(X,D) \ar[d]^{R\Gamma} \\
& \Kb(k)
}
\]
relating the morphisms $R\Gamma$ and $R\Gamma_{dR}$.
\end{itemize}
\end{theorem}

\begin{rmk}
According to Thomason--Trobaugh \cite[Theorem 2.6.3(d)]{MR1106918}, every $f\colon Y \to X$ as in Definition \ref{defi:infinitelynear} induces an equivalence $K$-theory spectra
$Lf^*\colon \Kb(X,D) \to^{\simeq} \Kb(Y,D).$
Thomason--Trobaugh's definition of \emph{isomorphisms infinitely near $D$} in \cite[Definition 2.6.2.1]{MR1106918} is different from our Definition \ref{defi:infinitelynear}. Yet, these two definitions are equivalent as is shown in \cite[Lemma-Definition 2.6.2.2]{MR1106918}.
In the light of Thomason--Trobaugh's result, the excision property (a) in Patel's Theorem \ref{thm:patel} can therefore be accurately described as stating that the epsilon factor $\Ec^{P}_{\nu}(M)$ of a holonomic $D$-modules depends only on the geometry of $X$ and $M$ near $D$.
\end{rmk}

If $D$ is proper, we may consider the composition $R\Gamma\circ \Ec_{\nu}^P\colon  \Kb(D_X,{\nu}) \to \Kb(k)$. Post-composing this with the graded determinant map $\Kb(k) \to \Pic^{\Zb}$, we define a de Rham epsilon line $\varepsilon^{P}_{\nu}(X,M)$. Furthermore, by virtue of the definition of algebraic $K$-theory, we have the following.

\begin{lemma-definition}\label{lemma-defi:P_beta}
For every short exact sequence of holonomic $D$-modules 
$$\xi\colon M' \hookrightarrow M \twoheadrightarrow M''$$
we have an isomorphism of graded lines
$$\beta^P_{\xi}\colon \varepsilon^P_{\nu}(X,M) \simeq \varepsilon^P_{\nu}(X,M') \otimes \varepsilon^P_{\nu}(X,M''),$$
such that for every diagram
\[
\xymatrix{
M \ar@{^(->}[r] & N \ar@{->>}[d] \ar@{^(->}[r] & P \ar@{->>}[d] \\
& N/M \ar@{^(->}[r]  & P/M \ar@{->>}[d] \\
& & P/N
}
\]
we have a commutative diagram of isomorphisms of graded lines
\[
\xymatrix{
\varepsilon^P_{\nu}(X,P) \ar[r] \ar[d] & \varepsilon^P_{\nu}(X,N) \otimes \varepsilon^P_{\nu}(X,P/N) \ar[d] \\
\varepsilon^P_{\nu}(X,M) \otimes \varepsilon^P_{\nu}(X,P/M) \ar[r] & \varepsilon^P_{\nu}(X,M) \otimes \varepsilon^P_{\nu}(X,N/M) \otimes \varepsilon^P_{\nu}(X,N/M)^{-1} \otimes \varepsilon^P_{\nu}(X,P/M). 
}
\]
\end{lemma-definition}

If $\dim X = 1$, $D \subset X$ a (reduced) divisor, we will show in Theorem \ref{thm:comparison} that Patel's epsilon line satisfies 
$$\varepsilon^{P}_{\nu}(X,M) \simeq \varepsilon^{BBE}_{\nu}(X,M),$$
and that with respect to this equivalence, $\beta^P_{\xi}$ corresponds to $\beta^{BBE}_{\xi}$.

In the remainder of this subsection we will sketch the main steps of Patel's construction. The starting point is a theorem of Quillen about the $K$-theory of $D$-modules on a smooth $k$-scheme $X$. We denote by $\FMod_{\coh}(D_X)$ the quasi-abelian category of $D_X$-modules with a good filtration. The associated graded defines an exact functor 
$$\gr\colon \FMod_{\coh}(D_X) \to \Coh(T^*X).$$

\begin{proposition}[Quillen]
There exists a morphism of spectra $Q\colon \Kb(D_X) \to \Kb(T^*X)$, such that the diagram
\[
\xymatrix{
\Kb(D_X) \ar[rd] & \Kb(\FMod_{\coh}(D_X)) \ar[l] \ar[d]^{\gr} \\
& \Kb(T^*X)
}
\]
commutes.
\end{proposition}

The key result underlying Patel's epsilon factors is a refinement of Quillen's result on filtered rings. Let $S \subset T^*X$ be a closed subset. We denote by $\Perf(D_X,S)$ the derived $\infty$-category of perfect complexes of $D$-modules on $X$ with singular support contained in $S$.

\begin{proposition}[Patel]
There exists a morphism of spectra $Q_S\colon \Kb(D_X,S) \to \Kb(T^*X,S)$ fitting into a commutative square
\[
\xymatrix{
\Kb(D_X,S) \ar[d] \ar[r] & \Kb(T^*X,S) \ar[d] \\
\Kb(D_X) \ar[r] & \Kb(T^*X).
}
\] 
\end{proposition}

We denote by $U = X \setminus D$. The $1$-form $\nu$ defines a section $\nu\colon U \to T^*X$ which does not intersect $S$ (by assumption). Let $\pi\colon T^*X \to X$ be the canonical projection.

The identity $\pi \circ \nu = \id_U$ gives rise to a commutative diagram of spectra
\[
\xymatrix{
\Kb(X) \ar[rd] \ar[r]^{\pi^*} & \Kb(T^*X) \ar[d]^-{\nu^*} \\
& \Kb(U).
}
\]
Since $X$ is regular and $\pi\colon T^*X \to X$ is Zariski-locally a fibration in to affine spaces, the induced morphism of spectra $\pi^*\colon \Kb(X) \to^{\simeq} \Kb(T^*X)$ is an equivalence. 
Therefore we also have a commutative diagram
\[
\xymatrix{
\Kb(X) \ar[rd] & \Kb(T^*X) \ar[l]_{(\pi^*)^{-1}} \ar[d]^-{\nu^*} \\
& \Kb(U).
}
\]
In particular, we get a commutative square
\[
\xymatrix{
\Kb(T^*X) \ar[r] \ar[d]_-{(\pi^*)^{-1}} & \Kb(T^*U) \ar[d]^-{\nu^*} \\
\Kb(X) \ar[r] & \Kb(U).
}
\]
This induces a morphism between the fibres (of the rows) $\phi_{\nu}\colon \Kb(T^*X,S) \to \Kb(X,D)$.

\begin{definition}[Patel]
The morphism $\Ec_{\nu}^P$ is defined to be $[-n] \circ \phi_{\nu} \circ Q_S\colon \Kb(D_X,{\nu}) \to \Kb(X,D)$, where $n = \dim X$.
\end{definition}

The excision property is obvious from the definitions and the product formula the content of \cite[Corollary 3.23]{MR2981817}.

\subsection{Comparison}

This subsection is devoted to verifying that Patel's epsilon line for curves agrees with Beilinson--Bloch--Esnault's epsilon line.

\begin{theorem}\label{thm:comparison}
\begin{enumerate}
\item[(a)]
For a holonomic $D$-module $M$ with singular support $S$ on $X$ we have an isomorphism of graded lines
$$\alpha_{M}\colon \varepsilon^{P}_{\nu}(X,M) \to^{\simeq} \varepsilon^{BBE}_{\nu}(X,M)[-1],$$
where $\nu(U) \cap S = \emptyset$.
\item[(b)]
For every short exact sequence of holonomic $D$-modules on $X$
$$\xi\colon M' \hookrightarrow M \twoheadrightarrow M'' $$
we have a commutative diagram
\[
\xymatrix{
\varepsilon^P_{\nu}(X,M) \ar[r]^{\alpha_M} \ar[d]_{\beta_{\xi}^P} & \varepsilon^{BBE}_{\nu}(X,M) \ar[d]^{\beta_{\xi}^{BBE}}[-1] \\
\varepsilon^P_{\nu}(X,M') \otimes \varepsilon^P_{\nu}(X,M'') \ar[r]^{\alpha_{M'} \otimes \alpha_{M''}} & \varepsilon^{BBE}_{\nu}(X,M')[-1]\otimes \varepsilon^{BBE}_{\nu}(X,M'')[-1].
}
\]
of isomorphisms of graded lines, where the isomorphisms $\beta^{BBE}_{\xi}$ and $\beta^P_{\xi}$ are specified by Lemma-Definition \ref{lemma-defi:BBE_beta} and Lemma-Definition \ref{lemma-defi:P_beta}.
\item[(c)] Let $X$ be proper (in addition to smooth) and $M$ a holonomic $D$-module on $X$. There is a commutative diagram
\[
\xymatrix{
\varepsilon^P_{\nu}(X,M) \ar[r]^{\alpha_M} \ar[rd] & \varepsilon^{BBE}_{\nu}(X,M)[-1] \ar[d] \\
& \grdet(R\Gamma_{dR}(X,M))
}
\]
and this isomorphism is compatible with short exact sequences of holonomic $D$-modules.
\end{enumerate}
\end{theorem}

The proof of this result will be given below. At first we link Patel's $Q_S$ and Quillen's $Q$ with Deligne's good pairs. This is the content of Lemma \ref{lemma:missing_link} below.

We denote by $\pi_{\leq 1}\colon \Grpd \to \mathsf{Grpd}$ the functor sending a space $X$ to the Poincar\'e groupoid consisting of points of $X$ and homotopy classes of paths between points. Similarly, we denote by $\pi_{\leq 2}$ the functor sending a space $X$ to the Poincar\'e $2$-groupoid, consisting of points, paths between points, and homotopy classes of homotopies between paths. 

In the following we denote by $[A \to B]$ a chain complex where $A$ is supported in degree $-1$ and $B$ in degree $0$.

\begin{lemma-definition}\label{lemma-defi:comparison}
\begin{enumerate}
\item[(a)] Let $X$, $U$, and $\Ec = (E,\nabla)$ be as in Situation \ref{situation:BBE}. There is a morphism
$$\gamma_0\colon \pi_{\leq 2}K(\Loc(U)) \to \pi_{\leq 2}K(X),$$
such that for a good lattice pair $(L,N)$ for $\Ec$ we have
$$\gamma_0(\Ec) \simeq [L \to^0 N \otimes \Omega_X^1(D)].$$ 

\item[(b)] For $\nu \in \Omega_X^{1,\times}(U)$ we have a morphism
$$\gamma_{\nu}\colon \pi_{\leq 2}K(\Loc(U)) \to \pi_{\leq 2}K(X,D),$$
such that the square
\[
\xymatrix{
\pi_{\leq 2}K(\Loc(U)) \ar[r]^{\gamma_{\nu}} \ar[d] & \pi_{\leq 2}K(X,D) \ar[d] \\
\pi_{\leq 2}K(\Loc(U)) \ar[r]^{\gamma_{0}} & \pi_{\leq 2}K(X) 
}
\]
commutes, and for a good lattice pair $(L,N)$ for $\Ec$ we have 
$$\gamma_{\nu}(\Ec) \simeq [L \to^{\nu} N \otimes \Omega_X^1(D)].$$
\end{enumerate}
\end{lemma-definition}

\begin{proof}
(a): We check first that there's a well-defined map $\Loc(U)^{\times} \to \pi_{\leq 2}K(X)$ sending $\Ec \in \Loc(U)$ to $[L \to^0 N \otimes \Omega_X^1(D)].$

\begin{claim}\label{claim:nu0}
Let $(L_1,N_1)$, and $(L_2,N_2)$ be good lattices for $\Ec$. There exists a homotopy $q_{12}$ in $\pi_{\leq 2}K(X)$ between
$[L_1 \to^0 N_1 \otimes \Omega_X^1(D)]$ and $[L_2 \to^0 N_2 \otimes \Omega_X^1(D)].$ Furthermore, given a third good lattice pair $(L_3,N_3)$ we have 
$$q_{12} \cdot{} q_{23} \simeq q_{13}.$$
This construction is compatible with quadruples of good lattice pairs.
\end{claim}
\begin{proof}
This proof is a facsimile of the proof of the existence of good epsilon lines (with the additional dimension of taking the $2$-categorical nature of $\pi_{\leq 2}$ into account). Without loss of generality (since the poset of good lattice pairs is filtered) we may assume that $(L_1,N_2) \subset (L_2,N_2) \subset (L_3,N_3)$. Using the $H$-group structure on $\pi_{\leq 2}K(X)$ we see that it suffices to construct a homotopy between $0 \in \pi_{\leq 2}K(X,Z)$ and $[L_{i+1}/L_i \to^0 N_{i+1}/N_i \otimes \Omega_X^1(D)]$.

By virtue of the definition of good lattice pairs, the complexes
$$[L_{i+1}/L_i \to^{\nabla} N_{i+1}/N_i \otimes \Omega_X^1(D)]$$
are acyclic for $i=1,2$. The map $\nabla$ is not $\Oo_X$-linear. However, we observe that $K(X,Z)$ is equivalent to the $K(\Sh^{f}_{X,Z}(k))$, where $\Sh^f_{X,Z}$ denotes the abelian category of sheaves of $k$-vector spaces on $X$ with support $Z$. In $\pi_{\leq 1}K(\Sh^f_{X,Z}(k))$ we have a homotopy
$$[L_{i+1}/L_i \to^{\nabla} N_{i+1}/N_i \otimes \Omega_X^1(D)] \simeq [L_{i+1}/L_i] \ominus [N_{i+1}/N_i \otimes \Omega_X^1(D)] \simeq [L_{i+1}/L_i \to^0 N_{i+1}/N_i \otimes \Omega_X^1(D)],$$
where $\ominus$ denotes subtraction with respect to the $H$-group structure on $\pi_{\leq 2}K(X)$. The left hand side is represented by an acylic complex, and therefore homotopic to $0$.
This concludes the proof of the claim.
\end{proof}
It remains to check that there exists a map $\pi_{\leq 2}K(\Loc(U)) \to \pi_{\leq 2}K(X)$, such that the diagram
\[
\xymatrix{
\Loc(U)^{\times} \ar[rd] \ar[d] & \\
\pi_{\leq 2}K(\Loc(U)) \ar[r] & \pi_{\leq 2}K(X)
 }
\]
commutes. This follows at once from the behaviour of lattice pairs with respect to short exact sequences.
\begin{claim}\label{claim:ses}
A short exact sequence of flat connections $\Ec' \hookrightarrow \Ec \twoheadrightarrow \Ec''$ on $U$ can be lifted to a short exact sequence of good lattice pairs:
\[
\xymatrix{
M' \ar@{^(->}[r] \ar[d]_{\nabla} & M \ar@{->>}[r] \ar[d]_{\nabla} & M'' \ar[d]_{\nabla} \\
N' \otimes \Omega^1_X(D) \ar@{^(->}[r] & N \otimes \Omega^1_X(D) \ar@{->>}[r] & N'' \otimes \Omega^1_X(D).
}
\]
\end{claim}
\begin{proof}
By virtue of formal descent we may assume that $X$ is a trait. Without loss of generality we assume that $X = \Spec k[[t]]$ and $U = \Spec F$ where $F = k((t))$. For split short exact sequences the assertion is obvious. It suffices therefore to prove the claim when $\Ec'$ and $\Ec''$ are indecomposable and the short exact sequence is non-split. It then follows from the Levelt--Turritin decomposition (see \cite{levelt}) that there exists a finite \'etale morphism $q\colon \Spec F' \to \Spec F$ (of generic points of traits), such that the short exact sequence is given by the push-forward $q_*$ applied to the short exact sequence 
$$\mathcal{L}\otimes (\Ec_{(n-1)} \hookrightarrow \Ec_{(n)} \to \Ec_{(i)})$$
where $\Ec_{(n)}$ is the flat connection $(\Oo_{F'}, d + \frac{J_{(n)}}{z}dz)$ (where $J_{(n)}$ is an $(n\times n)$-Jordan block), and $\mathcal{L}$ is a rank $1$ flat connection on $\Spec F'$.

We can then construct the short exact sequence of good lattice pairs by pushing-forward the pairs $(L^{\oplus \ell},L(\mathsf{irr}(\mathcal{L}))^{\oplus \ell})$ for $\ell = n-i,n,i$, where $L \subset \mathcal{L}$ is a Deligne lattice and $\mathsf{irr}(\mathcal{L})$ denotes the irregularity of $\mathcal{L}$.
\end{proof}

\begin{rmk}
The maps $\gamma_0$ and $\gamma_{\nu}$ can also be defined on the full $K$-theory spectrum. However, this is technical and more than what we need for the purpose of this subsection.
\end{rmk}

This concludes the construction of the map $\gamma_0$. We now briefly turn to $\gamma_{\nu}$: as in the proof of Claim \ref{claim:nu0} one verifies that there is a well-defined map $\Loc(U)^{\times} \to \pi_{\leq 2}K(X,Z)$ given by $\Ec \mapsto [L_1 \to^{\nu} N_1 \otimes \Omega_X^1(D)]$. At first we remark that the complex $[L_1 \to^{\nu} N_1 \otimes \Omega_X^1(D)]$ is acyclic when restricted to $U$, since $\nu(U) \cap S = \emptyset$. Therefore it defines indeed a point in $\pi_{\leq 2}K(X,Z)$. Furthermore, we may assume that $(L_1,N_2) \subset (L_2,N_2) \subset (L_2,N_3)$. As in the proof of Claim \ref{claim:nu0} we see that in $\pi_{\leq 1}K(\Sh^f_{X,Z}(k))$ 
$$[L_2/L_1 \to^{\nu} N_2/N_1 \otimes \Omega_X^1(D)] \simeq [L_2/L_1] \ominus [N_2/N_1 \otimes \Omega^1_X(D)] \simeq [L_2/L_1 \to^{\nabla} N_2/N_1 \otimes \Omega_X^1(D)] \simeq 0.$$
The compatibility of good lattice pairs with short exact sequences (Claim \ref{claim:ses}) implies the assertion.
\end{proof}

\begin{lemma}\label{lemma:missing_link}
There are commutative diagrams of Picard groupoids
\[
\xymatrix{
\pi_{\leq 2}K(\Loc(U)) \ar[r]^-{(\pi^*)^{-1}} \ar[rd]_{\gamma_{0}} & \pi_{\leq 2}K(X,D) \ar@{=}[d] \\
& \pi_{\leq 2}K(X,D),
}
\]
and 
\[
\xymatrix{
\pi_{\leq 2}K(\Loc(U)) \ar[r]^-{\Ec_{\nu}} \ar[rd]_{\gamma_{\nu}} & \pi_{\leq 2}K(X,D) \ar@{=}[d] \\
& \pi_{\leq 2}K(X,D).
}
\]
\end{lemma}

\begin{proof}
We will verify these statements after formal completion at $D$ (that is, after replacing $X$ by the disjoint union of traits $\widehat{X}_D$). 
For the second diagram this is sufficient by the excision property of $\Ec_{\nu}$ and $\gamma_{\nu}$. We can use the Levelt--Turritin decomposition (see \cite{levelt}) to analyse the morphism $\Ec_{\nu}\colon \pi_{\leq 1}K(\Loc(U)) \to K(X,D)$.
Henceforth we assume that $X$ is a trait and $U = \Spec F$ where $F = k((t))$. There exists a finite \'etale extension $q\colon\Spec F' \to \Spec F$, such that in $\pi_{\leq 1}(\Loc(U))$ we have a homotopy 
$$\Ec \simeq q_*(\mathcal{L})^{\oplus n}$$
where $\mathcal{L}$ is a rank $1$ connection on $\Spec F'$. We observe that $Q_S$ and $\Ec_{\nu}$ and $\gamma_{\nu}$ commute with the pushforward $q_*$ (as good lattice pairs do, and by virtue of \cite[Corollary 3.30]{MR2981817}). The $D$-module $j_*\mathcal{L}$ on $\Spec \Oo_{F'} = \Spec k'[[t']]$ is endowed with a good filtration 
$(j_*\mathcal{L})_{\leq m} = \mathcal{O}_{F'}(m (i+1)),$
where $i$ denotes the irregularity of $\mathcal{L}$. A simple computation using this filtration to compute the value of $Q$, concludes the proof of the second assertion.

For the first diagram we use the following well-known property of algebraic $K$-theory.
\begin{claim}[Formal descent for algebraic $K$-theory]
The commutative square of spectra
\[
\xymatrix{
K(X) \ar[r]^{j^*} \ar[d] & K(U) \ar[d] \\
K(\widehat{X}_D) \ar[r]^{\widehat{j}^*} & K(\widehat{X}_D \times_X U)
}
\]
is cartesian.
\end{claim}
\begin{proof}
It suffices to show that the induced map of fibres $\fib(j^*) \to \fib(\widehat{j}^*)$ is an equivalence. Thomason--Trobaugh's localisation theorem implies that this map is homotopic to the natural morphism 
$$K(X,D) \to K(\widehat{X}_D,D).$$
According to \cite[Theorem 2.6.3(d)]{MR1106918} this is an equivalence.
\end{proof}
In order to conclude the assertion, we have to show that the following commutative diagrams of $2$-Picard groupoids are equivalent:
\[
\xymatrix{
\pi_{\leq 2}K(\Loc(U)) \ar@{-->}[rd] \ar[r]^{\gamma_0} \ar[d]_{\gamma_0} & \pi_{\leq 2}K(U) \ar[d]^{f^*} &  & \pi_{\leq 2}K\Loc(U) \ar[r]^{(\pi^*)^{-1}} \ar@{-->}[rd] \ar[d]_{(\pi^*)^{-1}} & \pi_{\leq 2}K(U) \ar[d]^{f^*} \\
\pi_{\leq 2} K(\widehat{X}_D) \ar[r]^{\widehat{j}^*} & \pi_{\leq 2}K(\widehat{X}_D \times_X U) & &\pi_{\leq 2} K(\widehat{X}_D) \ar[r]^{\widehat{j}^*} & \pi_{\leq 2}K(\widehat{X}_D \times_X U).
}
\]
In order to construct such an equivalence we ``cut" the commutative squares into two halves along the dashed arrow, as indicated in the diagram above. We will then compare these halves individually. We claim that we have a two equivalent commutative triangles of Picard $2$-groupoids
\[
\xymatrix{
\pi_{\leq 2}K(\Loc(U)) \ar[rd]_F \ar[r]^{\gamma_0} & \pi_{\leq 2}K(U) \ar[d]^{f^*} &  & \pi_{\leq 2}K\Loc(U) \ar[r]^{(\pi^*)^{-1}} \ar[rd]_F & \pi_{\leq 2}K(U) \ar[d]^{f^*} \\
 & \pi_{\leq 2}K(\widehat{X}_D \times_X U) & &  & \pi_{\leq 2}K(\widehat{X}_D \times_X U),
}
\]
where $F\colon D^b(\Loc(U)) \to \Perf(\widehat{X}_D \times_X U)$ is the exact functor sending $(E,\nabla)$ to $[E \to^0 E \otimes \Omega_X^1] \otimes_{\Oo_U} \widehat{\Oo}_{X,D}$. For the left hand side this follows from the definition of $\gamma_0$, for the right hand side this is a consequence of $(\pi^*)^{-1} \simeq i_0^*$, where $i_0 \colon X \to T^*X$ denotes the zero section. The same argument provides a homotopy between $\gamma_0$ and $(\pi^*)^{-1}$ which extends to a comparison of the two commutative diagrams. Mutatis mutandis we compare the remaining two commutating triangles.
\end{proof}

\begin{proof}[Proof of Theorem \ref{thm:comparison}]
For a holonomic $D$-module $M$ on $X$, there exists an open subset $j\colon U \hookrightarrow X$, such that $M|_U \simeq \Ec$ is a vector bundle with a flat connection. The fibre of the map $M \to j_*\Ec$ is a complex of $D$-modules with support on $D = X\setminus U$. Thus it suffices to prove the theorem for the $D$-module $j_*\Ec$.

Applying the determinant of cohomology, Lemma \ref{lemma:missing_link} yields a commutative diagram of Picard groupoids.
\[
\xymatrix{
\pi_{\leq 1}K(\Loc(U)) \ar[r]^-{\varepsilon^P_{\nu}} \ar[rd]_{\grdet(R\Gamma(\gamma_{\nu}))} & \grPic(k) \ar@{=}[d] \\
& \grPic(k).
}
\]
It remains to show that $\grdet(R\Gamma(X,\gamma_{\nu}))$ is isomorphic to $\tilde{\varepsilon}^{BBE}(X,\Ec)$. This is a consequence of Corollary \ref{cor:good}. 
\end{proof}

\subsection{The epsilon connection}

In \cite{MR1988970} Beilinson--Bloch--Esnault study de Rham epsilon lines in a more general setting. They admit $S$-families of smooth curves $X$ (where $S$ is a $k$-scheme), and consider $S$-families of irregular flat connections over $U_S \subset X_S$, satisfying an admissibility condition. Following \emph{loc. cit.} we refer to such $S$-families as being \emph{epsilon-nice} The epsilon line is computed with respect to a relative $1$-form $\nu \in \Omega^1_{X_S/S}(U_S)$. Studying the variation of epsilon lines in dependence of $\nu$, the so-called epsilon connection on $\tilde{\varepsilon}_{\nu}(E,\nabla)$ is defined. Recasting flat connections in terms of crystals, the epsilon connection amounts to the following invariance property:
\begin{theorem}[Beilinson--Bloch--Esnault]
Let $(E,\nabla)$ be an \emph{epsilon-nice} $S$-family of flat connections on $U_S \subset X_S \to S$. Let $A$ be a commutative $k$-algebra, $\Spec A \to S$, such that we have two sections $\nu_1,\nu_2\in \Omega_{X_A/A}^1(U)$. If $(\nu_1)_{A^{\red}} = (\nu_2)_{A^{\red}}$ then there exists an isomorphism 
$$c_{21}\colon \tilde{\varepsilon}_{\nu_1}(E,\nabla) \simeq \tilde{\varepsilon}_{\nu_2}(E,\nabla).$$
For every triple $\nu_1,\nu_2,\nu_3 \in \Omega_{X_A/A}^1(U)$ satisfying $(\nu_1)_{A^{\red}} = (\nu_2)_{A^{\red}} = (\nu_3)_{A^{\red}}$ we have $c_{32} \circ \nu_{21} = c_{31}$.
\end{theorem}

In this subsection we recall the construction of this epsilon-crystal in the special case of a constant family of curves $X_k \times_{\Spec k} S$, and of an irregular flat connection $(E,\nabla)/U_k$. We restrict to this case, in order to simplify notation, and as our main interest lies in varying the rational $1$-form $\nu$.

\begin{situation}\label{situation:relativeBBE}
Let $X$, $U$, $M$, $(E,\nabla)$ be as in Situation \ref{situation:BBE}. We let $A$ be a commutative $k$-algebra, and $\nu \in \Omega^1_{X_A/A}$ a nowhere vanishing section. 
\end{situation}

\begin{lemma-definition}\label{lemma-defi:relativeBBE} 
\begin{enumerate}
%\item[(a)] A \emph{pair of lattices} for $(E,\nabla)$ are locally free $\Oo$-coherent subsheaves $L, N\hookrightarrow j_*E$, such that $L|_U = E$, $N|_U = E$ and satisfying $\nabla(L) \subset N \otimes \Omega_X^1$.
%
%\item[(b)] We denote by $\DR(L,N))$ the complex of sheaves
%$[L \to^{\nabla} N \otimes \Omega_X^1].$
%
\item[(a)] We choose a pair of good lattices $(L,N)$ for $(E,\nabla)$ and define
$$\tilde{\varepsilon}_{\nu}(E,\nabla) = \grdet(L_A:\nu^{-1}(N_A \otimes \Omega_X^1)) \otimes \left(\grdet(R\Gamma(\DR(L,N))\right)_A^{-1} \otimes (\grdet(R\Gamma_{dR}(U,E,\nabla)))_A.$$

\item[(b)] The de Rham epsilon line of $M$ is defined to be the graded line $\varepsilon_{\nu}(M) = \tilde{\varepsilon}_{\nu}(E,\nabla) \otimes \grdet(R\Gamma_{dR,D}(M)).$
\end{enumerate}
\end{lemma-definition}

\begin{proof}
Mutatis mutandis this is the same argument as in Lemma-Definition \ref{lemma-defi:BBE} to show that the de Rham epsilon line is well-defined.
\end{proof}

In order to study the dependence of the epsilon line on infinitesimal changes of $\nu$; we assume that $X$ is a trait. 

\begin{situation}\label{situation:relativeBBE2}
We assume that $X = \Spec k'[[t]]$ is a trait. Let $\nu_1, \nu_2, \nu_3 \in \Omega^1_{A((t))/A}$, such that $(\nu_1)_{A^{\red}} = (\nu_2)_{A^{\red}} = (\nu_3)_{A^{\red}}$. We denote by $f$ the element in $A((t))^{\times}$, such that $\nu_2 = f\nu_1$.
\end{situation}

We denote by $\grdet(f)$ the relative determinant
$$\grdet(f^{-1}L:L)$$
where $L \subset j_*E$ is an arbitrary lattice. The transitivity property (Lemma \ref{lemma:transitivity}) yields
$$\tilde{\varepsilon}_{f\nu}(\Ec) \simeq \tilde{\varepsilon}_{\nu}(\Ec) \otimes \grdet(f).$$
Using a representation-theoretic argument based on Clifford theory, Beilinson--Bloch--Esnault showed in \cite[3.5]{MR1988970} that the super line associated to the graded line $\grdet(f)$ admits a $\mu_2$-reduction, and hence in particular a crystal structure over $\Spec A$. This induces a crystal structure on the epsilon line.

\section{Generalisations of de Rham epsilon factors}

The goal of this section is to generalise Patel's construction (see \cite{MR2981817}, and \ref{Patel} for a summary) in two different directions. As it is, the epsilon factor takes values in $\Kb(X,D)$. If neither $X$ nor $D$ are proper, we can not apply the derived pushforward $R\Gamma(X,-)$ to obtain an epsilon factor in $\Kb(k)$ and thus an epsilon line.

Our first goal is therefore to modify Patel's definition, so that we obtain epsilon factors taking values in $\Kb(X,Z)$ where $Z \subset X$ is a zero-dimensional closed subset of $X$. This requires working with $n = \dim X$ rational differential forms $\nu_1,\dots,\nu_n$ whose domains of definition cover $X \setminus Z$ and are sufficiently generic (in particular, rationally linearly independent). This formalism of epsilon factors hinges on the $\Ab^1$-homotopy invariance of $K$-theory.

Secondly, we want to study how these new epsilon factors vary in dependence of the $n$-tuple of $1$-forms $(\nu_1,\dots,\nu_n)$. We therefore replace the field $k$ by a commutative $k$-algebra $A$, and for every holonomic $D$-modules $M$ on $X$ we define a so-called epsilon connection over the space ${\bf\Omega}$ of admissible $n$-tuples of $1$-forms $(\nu_1,\dots,\nu_n)$. This is the content of Section \ref{epsilon}.

\subsection{Epsilon factors supported on points}

In general, we expect Patel's epsilon factor to be supported on a closed subset $D \subset X$ of codimension $1$. Using several $1$-forms at once, we can replace $D$ by a smaller closed subset $Z$, even of codimension $n$.

Let $I$ be a non-empty finite ordered set. We denote by $\Ab^{\Delta_{I}}$ the affine space of dimension $|I|-1$ defined by the linear equation 
$$\sum_{i \in I} \lambda_i = 1.$$

\begin{situation}\label{situation:tuple}
Let $X$ be a smooth $n$-dimensional $k$-variety and $Z \subset X$ a closed subset. Let $S \subset T^*X$ be a closed subset. We denote the complement of $Z$ by $U$. Consider an open covering $U = \bigcup_{i=1}^m U_i$, and for each $i=1,\dots, m$ a nowhere vanishing $1$-form $\nu_i \in \Omega_X^1(U_i)$, such that for each ordered subset $\{i_1 < \cdots < i_{\ell}\} \subset \{1,\dots, m\}$ we have for $U_{i_1\dots i_{\ell}} = \bigcap_{j=1}^{\ell} U_{i_j}$ that the image of the morphism $\nu_{\Delta}((U_{i_1 \dots i_{\ell}})_A) \cap S_A = \emptyset$, where $\nu_{\Delta}\colon U_{i_1\dots i_{\ell}} \times \Ab^{\Delta} \to T^*X$, $\nu_{\Delta}(\lambda_1,\dots,\lambda_{\ell}) =  \sum_{j= 1}^{\ell} \lambda_j \nu_{i_j}$.
\end{situation}

Henceforth we denote an $m$-tuple of forms $(\nu_1,\dots,\nu_m)$ as in Situation \ref{situation:tuple} by $\underline{\nu}$. Mimicking Patel's definition, we introduce for every closed subset $S \subset T^*X$ a map $\phi_{\underline{\nu}}\colon \Kb(T^*X,S) \to \Kb(X,Z)$ in Lemma \ref{lemma:phi}, and define $\Ec_{\underline{\nu}} = \phi_{\underline{\nu}} \circ Q_S$.

\begin{lemma-definition}\label{lemma:phi}
There exists a morphisms $(\underline{\nu})^*\colon \Kb(T^*X \setminus S) \to \Kb(U)$, such that for all $i=1,\dots,m$ we have a commutative diagram
\[
\xymatrix{
\Kb(T^*X \setminus S) \ar[r]^-{(\underline{\nu})^*} \ar[rd]_{\nu_i^*} & \Kb(U) \ar[d] \\
& \Kb(U_i).
}
\]
Furthermore, this morphism satisfies $(\underline{\nu})^*\circ\pi^* \simeq \id$.
\end{lemma-definition}

\begin{proof}[Sketch]
As a warm-up we treat the case $m=2$. The idea is to show that on $U_{12} = U_1 \cap U_2$ we can construct a linear homotopy $\nu_t\colon U_{12} \times \Ab^1 \to T^*U_{12}$ between the sections $\nu_1|_{U_{12}}$ and $\nu_2|_{U_{12}}$ by $\nu_t = (1-t) \cdot{} \nu_1 + t \cdot{} \nu_2$ for $t \in \Ab^1$. In this step we use that $\nu_1$ and $\nu_2$ are linearly independent on $U_{12}$.
\begin{claim}
There is a commutative diagram
\[
\xymatrix{
\Kb(T^*X) \ar[r]^{\nu_1^*} \ar[d]_{\nu_2^*} & \Kb(U_1) \ar[d] \\
\Kb(U_2) \ar[r] & \Kb(U_{12})
}
\]
of spectra.
\end{claim}
\begin{proof}
Since this diagram is a square, it suffices to verify that it is homotopy commutative, that is, $(\nu_1|_{U_{12}})^* \simeq (\nu_2|_{U_{12}})^*$. We have a commutative diagram of regular schemes
\[
\xymatrix{
U_{12} \times \{0\} \ar[rd]_{i_0} \ar@/^/[rrd]^{\nu_0} & & \\
& U_{12} \times \Ab^1 \ar[r]^{\nu_t} & T^*X \\
U_{12} \times \{1\} \ar[ru]^{i_1} \ar@/_/[rru]_{\nu_1} & &
}
\]
Since the $K$-theory of regular schemes is $\Ab^1$-invariant, and $i_0, i_1$ are sections of $\pi\colon U_{01} \times \Ab^1 \to U_{01}$, we have that pullback along $i_0$ and $i_1$ are inverse to $\pi^*$. Therefore, $\nu_0^* \simeq i_0^* \circ \nu_t \simeq (\pi^*)^{-1} \circ \nu_t \simeq i_1^* \circ \nu_t^* \simeq \nu_1^*$. 
\end{proof}

The proof of the general case follows the same idea. For $m \geq 3$ we have to produce a coherent system of linear homotopies. We don't explain this in detail, as a similar (and in fact more general) construction is performed in Subsection \ref{sub:relative}.

By Thomason--Trobaugh \cite[Theorem 8.1]{MR1106918}, algebraic $K$-theory satisfies Zariski descent, the system of coherent homotopies alluded to above yields the descent datum for a morphism $\Kb(T^*X \setminus S) \to \Kb(U)$. The assertion $(\underline{\nu})^*\circ(\pi^*) \simeq \id$ can be checked Zariski-locally, and therefore holds by construction.
\end{proof}

We now define $\Ec_{\underline{\nu}}\colon \Kb(D^{\nu}) \to \Kb(X,Z)$ by mimicking Patel's definition surveyed in \ref{Patel}. By construction we have a commutative diagram
\[
\xymatrix{
\Kb(T^*X) \ar[r] \ar[d]_-{(\pi^*)^{-1}} & \Kb(T^*X \setminus S) \ar[d]^-{\nu^*} \\
\Kb(X) \ar[r] & \Kb(U).
}
\]
This gives rise to a morphism of fibres (of the rows), that is, a map 
$$\phi_{\underline{\nu}}\colon\Kb(T^*X,S) \to \Kb_{X,Z}.$$

\begin{definition}
We define $\Ec_{\underline{\nu}}= Q_S \circ \phi_{\underline{\nu}}$.
\end{definition} 

Depending on how we choose $\underline{\nu}$ we may even assume that $Z$ is a zero-dimensional closed subscheme. In this case we can always apply $\grdet \circ R\Gamma$ and use this as a definition of an epsilon determinant. 

\begin{rmk}
The $\Ab^1$-invariance property for algebraic $K$-theory of regular schemes also plays an implicit role in Abe--Patel's \cite{abepatel}. In \emph{loc. cit.} Levine's coniveau tower  (\cite{Levine}) is used to prove a localization formula for epsilon-factors. Levine's construction is based on the $\Ab^1$-invariance property of the algebraic $K$-theory spectrum of a regular scheme.
\end{rmk}

\subsection{De Rham epsilon factors in families}

\subsubsection{A replacement for $\Ab^1$-homotopy}

When replacing $X/k$ by $X_A = X \times_{\Spec k} \Spec A$ we can no longer apply $\Ab^1$-invariance of $K$-theory. After all, the ring $A$ might not be regular. In particular, pullback along $\pi_A \colon T^*X_A \to X_A$ induces no longer an equivalence of $K$-theory spectra. The map $(\pi^*)^{-1}$ which is used in the construction of de Rham epsilon factors is therefore no longer available to us. Similarly, the linear homotopies alluded to in the discussion of Lemma-Definition \ref{lemma:phi} cannot be interpreted as $\Ab^1$-homotopies anymore.

As a substitute we can use homotopies indexed by projective spaces $\Pb^N$ with a modulus condition at infinity. This is reminiscent of the theory of reciprocity sheaves and motives with modulus (\cite{ksy}).

 We recall the following well-known result which provides the technical backbone for this strategy. Let $Z$ be a quasi-compact and quasi-separated scheme, $i\colon \Pb_Z^{N-1} \to \Pb_Z^N$ be the inclusion as divisor at $\infty$. Since $i$ is regular we have an induced map $i_*\colon \Kb(\Pb_Z^{N-1}) \to \Kb(\Pb_Z^N)$. We denote by $j\colon Z \to \Pb_Z^N$ a section, factorising through the open subscheme $\Pb^N_Z\setminus \Pb^{N-1}_Z$.

\begin{proposition}\label{prop:Phomotopy}
There is a cocartesian diagram of spectra
\[
\xymatrix{
\Kb(\Pb_Z^{N-1}) \ar[r]^{i_*} \ar[d] & \Kb(\Pb_Z^N) \ar[d]^{Lj^*} \\
0 \ar[r] & \Kb(Z).
}
\]
\end{proposition}

\begin{proof}
It is clear that the commutative diagram above exists (furthermore, commutativity of a square can be checked in the homotopy category).
According to Thomason--Trobaugh \cite[Theorem 7.3]{MR1106918}, the exact functor $$\Perf(Z)^{N+1} \to \Perf(\Pb^N_Z)$$ which sends $(V_0,\dots,V_n)$ to $\bigoplus_{i=0}^n V_i \otimes \Oo_{\Pb^N_Z}(-i)$ induces an equivalence $\Kb(Z)^{N+1} \to^{\simeq} \Kb(\Pb^N_Z)$.

By virtue of the locally free resolution
$$[\Oo_{\Pb^{N}_Z}(-1) \to \Oo_{\Pb^N_Z}] \to i_*\Oo_{\Pb^{N-1}_Z}$$
we see that we have a commutative diagram
\[
\xymatrix{
\Kb(\Pb^{N-1}_Z) \ar[r]^{i_*} \ar[d]_{\simeq} & \Kb(\Pb^N_Z) \ar[d]^{\simeq} \\
\bigoplus_{i=0}^{N-1} \Kb(Z) \ar[r]^{\alpha} & \bigoplus_{i=0}^N \Kb(Z) ,
}
\]
where $\alpha$ is homotopic to $(e_0,\dots,e_{n-1}) \mapsto (e_0,e_1 - e_0,\dots,e_{n-1}-e_{n-2},-e_{n-1})$. The cofibre is therefore equivalent to a copy of $\Kb(Z)$. It is given by the codiagonal morphism $\beta \colon \bigoplus_{i=0}^N \Kb(Z) \to \Kb(Z)$ which sends $(e_0,\dots,e_n)$ to $e_0 + \cdots + e_n$. Using the description of the $K$-theory of $\Pb^N_Z$ of \emph{loc. cit.}, one sees that 
\[
\xymatrix{
\Kb(\Pb^{N}_Z) \ar[r]^{Lj^*} \ar[d]_{\simeq} & \Kb(A) \ar[d]^{\simeq} \\
\bigoplus_{i=0}^{N-1} \Kb(Z) \ar[r]^{\beta} & \bigoplus_{i=0}^N \Kb(Z) 
}
\]
commutes. This concludes the proof of the assertion.
\end{proof}

\begin{construction}\label{const:homotopy}
Let $Z$ be a quasi-compact and quasi-separated scheme, and $s_0,\dots,s_n\in (\Pb^N_Z \setminus \Pb^{N-1})(Z)$ an $(n+1)$-tuple of sections. Then we construct an $n$-simplex of spectra 
$$\sigma_{s_0,\dots,s_n}\colon\Delta^n \to \Spectra,$$
such that for every $1$-face $\Delta^1_i = \{0,i\} \subset \Delta^n$, the $1$-simplex $\sigma_{s_0,\dots,s_n}|_{\Delta^1_i} = s_i^*\colon \Delta^1 \to \Spectra$ is the morphism of spectra $Ls_i^*\colon K(\Pb^N_Z) \to \Kb(Z)$, and for all other $1$-faces $\{i,j\} \subset [n]$ with $0 \neq i,j$, $\sigma_{s_0,\dots,s_n}|_{\Delta_{ij}} = \id_{\Kb_{Z}}$.  

Furthermore, for every given choice of $(n-1)$-simplices $\sigma_i$ for $(s_0,\dots,\widehat{s_i},\dots,s_n)$ (for all $i \in \{0,\dots,n\}$) satisfying the property above, we may assume that $\sigma_i$ is the $i$-th face of $\sigma_{s_0,\dots,s_n}$. 
\end{construction}

\begin{proof}[Construction]
This follows from \cite[Remark 1.1.1.7]{Lurie:ha}: if $\Cc$ is a stable $\infty$-category (in our case the $\infty$-category of spectra), then we denote by $\Ec$ the full subcategory of the $\infty$-category of squares $\Map((\Delta^1)^2,\Cc)$, consisting of all cofibre diagrams
\[
\xymatrix{
A \ar[r] \ar[d] & B \ar[d] \\
0 \ar[r] & C.
}
\]
The functor $\Psi\colon \Ec \to \Map(\Delta^1,\Cc)$ which sends such a cofibre diagram to the morphism $(A \to B) \in \Map(\Delta^1,\Cc)$, is an equivalence. That is, its fibres are contractible Kan complexes.

We now apply this property of stable $\infty$-category to the morphism of spectra 
$$\Kb(\Pb_Z^{N-1}) \to^{i_*} \Kb(\Pb_Z^N).$$
By virtue of Proposition \ref{prop:Phomotopy}, we have that for all $i=0,\dots,n$ a cofibre diagram
\[
\xymatrix{
\Kb(\Pb_Z^{N-1}) \ar[r]^{i_*} \ar[d] & \Kb(\Pb_Z^N) \ar[d]^{Ls_i^*} \\
0 \ar[r] & \Kb(Z),
}
\]
which belongs to $\Psi^{-1}(i_*)$. Contractibility of $\Psi^{-1}(i_*)$ yields the required $n$-simplices.
\end{proof}

\begin{corollary}\label{cor:Phomotopy}
Let $h\colon Z \times \Pb^1 \to Y$ be a morphism, we write $h \circ j_0 = h_0$ and $h \circ j_1 = h_1$. Then we have a homotopy of maps of spectra $h_0^*\simeq h_1^*\colon \Kb(Y) \to \Kb(X)$.
\end{corollary}

\begin{proof}
The diagram of Proposition \ref{prop:Phomotopy} is cocartesian and its bottom left corner is contractible (we call such diagrams cofibre diagrams). The $\infty$-category of cofibre diagrams is equivalent to the the $\infty$-category of morphisms $\Hom(\Delta^1,\C)$. In particular, we conclude that $Lj_x^*\simeq Lj_y^*$ for two rational points $x,y$ of $\Pb^1$. That is, $Lj_0^* \simeq Lj_1^*$.

For $i= 0,1$ we have a commutative diagram
\[
\xymatrix{
X \times \Pb^1 \ar[r]^h & Y \\
X, \ar[u]^{j_i} \ar[ur]_{h_i} &
}
\]
where $j_i\colon X \to \Pb^1_X$ denotes the section corresponding to the rational point $i \in \Pb^1_k$. We have
$$h_i^* \simeq h^* \circ j_i^*.$$
In the first paragraph of this proof we have shown that $j_0^* \simeq j_1^*$, and hence we obtain $h_0^* \simeq h_1^*$.
\end{proof}

We will also refer to a morphism $h\colon Z \times \Pb^1 \to Y$ as a $\Pb^1$-homotopy between $h_0$ and $h_1$.
%
%\begin{corollary}
%Let $h\colon X \times (\Pb^1)^n \to Y$ be a cubical $\Pb^1$-homotopy. Then for every pair of principal vertices of the cube $i,j$, we have $h_i^* \simeq h_j^*\colon \Kb(Y) \to \Kb(X)$.
%\end{corollary}

\subsubsection{Associated graded modules}

Let $Y$ be a scheme and $F\Ac$ a filtered quasi-coherent sheaf of algebras, such that $\gr \Ac$ is a commutative quasi-coherent sheaf of algebras on $Y$. Furthermore, we assume that for every $i \in \Nb$, the $i$-th filtered piece $F^i\Ac$ is a finitely generated locally free sheaf on $Y$. We denote by $V \to Y$ the relative spectrum of $\gr \Ac$. A sheaf of $\Ac$-modules is always understood to be a right $\Ac$-module.

\begin{definition}
The $K$-theory of perfect complexes of filtered $F\Ac$-modules will be denoted by $\Kb(F\Ac)$. The $K$-theory of perfect complexes of bifiltered modules (see \cite[p. 12]{MR2981817}) will be denoted by $K(FF\Ac)$. The functor $\Perf(F\Ac) \to^{\gr} \Perf(V)$ gives rise to a morphism of spectra $\Kb(F\Ac) \to \Kb(V)$.
\end{definition}

The following statement is proven exactly like \cite[Theorem 3.1.15(1)]{MR2981817}. It allows us to descend constructions in the algebraic $K$-theory of filtered $\Ac$-modules to the algebraic $K$-theory of unfiltered $\Ac$-modules.

\begin{lemma}[Patel]\label{lemma:bifilt}
Assume that every perfect complex of $\Ac$-modules as a finite resolution by sheaves of locally projective $\Ac$-modules. Let $S \subset V$ be a closed $\G_m$-invariant subset. We have a cocartesian diagram of spectra
\[
\xymatrix{
\Kb(FF\Ac,S) \ar[r] \ar[d] & \Kb(F\Ac,S) \ar[d] \\
\Kb(F\Ac,S) \ar[r] & \Kb(\Ac,S),
}
\]
where the downward morphism $\Kb(FF\Ac) \to \Kb(F\Ac)$ forgets the first filtration, and the rightward morphism $\Kb(FF\Ac) \to \Kb(F\Ac)$ forgets the second filtration.
\end{lemma}

The associated graded $\gr\colon \Perf(F\Ac) \to \Perf(V)$ factors naturally through $\Perf([V/\G_m])$. This observation is recorded in the right hand triangle of the following commutative diagram.

\begin{lemma}\label{lemma:associated_graded}
For a $\G_m$-invariant closed subset $S \subset V$ there exists a morphism of spectra $$Q_S\colon \Kb(\Ac) \to \Kb(V),$$ such that we have a commutative diagram
\[
\xymatrix{
\Kb(F\Ac,S) \ar[r]^-{\eqgr} \ar[d] \ar[rd]^{\gr} & \Kb([V/\G_m],[S/\G_m]) \ar[d]^{q^*} \\
\Kb(\Ac,S) \ar@{-->}[r]^-{Q_S} & \Kb(V,S)
}
\]
of spectra, where $q\colon V \to [V/\G_m]$ is the quotient map.
\end{lemma}

\begin{proof}
It suffices to construct the left hand commuting triangle. As in \cite{MR2981817} we consider the $K$-theory of perfect bifiltered $F\Ac$-modules with set-theoretic support of the associated graded module (with respect to either filtration) on $S \subset V$. By reducing to the case adjacent filtrations as in the proof of \cite[Theorem 3.1.5, 3.1.15(2)]{MR2981817}, one produces a commutative diagram
\[
\xymatrix{
\Kb(FF\Ac,S) \ar[r] \ar[d] & \Kb(F\Ac,S) \ar[d]^{\gr} \\
\Kb(F\Ac,S) \ar[r]^{\gr} & \Kb(V,S).
}
\]
Since $\Kb(\Ac,S)$ is equivalent to the pushout of spectra $\Kb(F\Ac,S) \sqcup_{\Kb(FF\Ac,S)} \Kb(F\Ac,S)$ (Lemma \ref{lemma:bifilt}), we obtain the required morphism $Q_S\colon \Kb(\Ac) \to \Kb(V)$.
\end{proof}

\subsubsection{Compactification}\label{compactification}

The purpose of this paragraph is to state and prove \ref{prop:compactification}. We set the stage by fixing notation.

\begin{situation}\label{situation:compactification}
Let $Y$ be a scheme and $\Ac/Y$ a quasi-coherent sheaf of algebras on $Y$ with a filtration $F_i\Ac$, such that the associated sheaf of graded algebras is isomorphic to $\Sym E$, where $E/Y$ is a vector bundle on $Y$ with total space denoted by $V$. Let $S \subset V$ be a $\G_m$-invariant closed subset. We consider a line bundle $L$ on $Y$ and denote the twist of $V$ (respectively $S$) by $L$ with $V_L$ (respectively $S_L$). Let $U \subset Y$ be an open subscheme, such that we have an isomorphism $L|_U \simeq \Oo_U$.
\end{situation}

\begin{proposition}\label{prop:compactification}
There is a commutative square of spectra
\[
\xymatrix{
\Kb(F\Ac,S) \ar[r] \ar[d] & \Kb(V_L,S_L) \ar[d] \\
\Kb(\Ac,S)  \ar[r] & \Kb(V \times_Y U, S \times_Y U) 
}
\]
where the vertical arrows are given by forgetting the filtration, respectively restriction to $U$.
\end{proposition}

The proof will be given at the end of this paragraph. We maintain the notation fixed in Situation \ref{situation:compactification}. The proof of the following statement is an amusing exercise and hence we will only be sketched.

\begin{lemma}\label{lemma:twist}
For every $L$ there exists an equivalence of quotient stacks $[V/\G_m] \simeq [V_L/\G_m].$
Furthermore, it induces an equivalence $[(V \setminus S)/\G_m] \simeq [(V_L \setminus S_L)/\G_m].$
\end{lemma}

\begin{proof}[Sketch]
We begin by constructing a morphism $V \to [V_L / \G_m]$. By definition of quotient stacks, it corresponds to a $\G_m$-equivariant morphism
$$f\colon V \times_Y P \to V_L$$
where $P$ is a $\G_m$-torsor and the $\G_m$ acts on the left hand side solely on the second component $P$. We choose $P$ to be the $\G_m$-torsor given by $L$ minus the zero section. We then have a natural $\G_m$-equivariant morphism $V \times_Y P \to V_L$ which sends a pair of sections $(s,t)$ to $s \otimes t$. This defines a morphism $[V/\G_m] \to [V_L/\G_m]$. Its inverse can be constructed by similar means, and the second assertion follows by inspection.
\end{proof}

\begin{proof}[Proof of Proposition \ref{prop:compactification}]
Exhibiting a commutative square of spectra is equivalent to defining a commutative square in the homotopy category of spectra. We claim that the following diagram in the homotopy category of spectra commutes:
\[
\xymatrix{
\Kb(F\Ac,S) \ar[r] \ar[d] & \Kb([V/\G_m],[S/\G_m]) \ar[r] \ar[d] & \Kb([V_L/\G_m],[S_L/\G_m]) \ar[d] \ar[r] & \Kb(V_L, S_L) \ar[ld] \\
\Kb(\Ac,S) \ar[r] & \Kb(V,S)   \ar[r] & \Kb(V \times_Y U,S\times_Y U). &
}
\]
The left hand square commutes by Lemma \ref{lemma:associated_graded}. The second square is obtained by applying the algebraic $K$-theory functor to the equivalence of Lemma \ref{lemma:twist}. The commuting triangle on the right is obtained from the isomorphism $V_L \times_Y U \simeq V \times_Y U$, stemming from the trivialisation of $L$.
\end{proof}

\subsubsection{The relative de Rham epsilon factor}\label{sub:relative}

Let $I$ be a non-empty finite ordered set. We denote by $\Ab^{\Delta_{I}}$ the affine space of dimension $|I|-1$ defined by the linear equation
$$\sum_{i \in I} \lambda_i = 1.$$

\begin{situation}\label{situation:tuple2}
Let $X$ be a smooth $n$-dimensional $k$-variety and $Z \subset X$ a closed subset. Let $S \subset T^*X$ be a closed subset. We denote the complement of $Z$ by $U$. Consider an open covering $U = \bigcup_{i=1}^m U_i$, and for each $i=1,\dots, m$ a section $\nu_i \in \Omega_{X_A/A}^1(U_i)$, such that for each ordered subset $\{i_1 < \cdots < i_{\ell}\} \subset \{1,\dots, m\}$ we have that for $U_{i_1\dots i_{\ell}} = \bigcap_{j=1}^{\ell} U_{i_j}$, the image of the morphism $\nu_{\Delta}((U_{i_1 \dots i_{\ell}})_A) \cap S_A = \emptyset$ where $\nu_{\Delta}\colon U_{i_1\dots i_{\ell}} \times \Ab^{\Delta} \to T^*X$, $\nu_{\Delta}(\lambda_1,\dots,\lambda_{\ell})= \sum_{j= 1}^{\ell} \lambda_j \nu_{i_j}$.
\end{situation}

The next construction is based on the $\Pb$-invariance property of algebraic $K$-theory. 
\begin{convention}
\begin{enumerate}
\item[(a)] We denote by $T^*X_A$ the base change $T^*X \times_{k} A$, or equivalently, the relative cotangent bundle $T^*(X_A/A)$.
\item[(b)] The notation $T^*X_A \boxtimes \Oo_{\Pb_A^m}(1)$ refers to the total space of the locally free sheaf $p_1^*\Omega^1_{X/A} \otimes p_2^*\Oo_{\Pb^m_A(1)}$, where $p_1$ and $p_2$ are the projections from $T^*X_A \times_A \Pb^m_A$ to $T^*X_A$ respectively $\Pb^m_A$.
\item[(c)] For a $\G_m$-equivariant subscheme $S \subset T^*X_A$ we write $S \boxtimes \Oo_{\Pb^m_A}(1)$ to denote the twist of $S \times \Pb_A^m$ by $p_2^*\Oo_{\Pb^m_A}(1)$ (in the sense of \ref{compactification}).
\item[(d)] We use the suggestive notation
$$(T^*X_A \setminus S) \boxtimes \Oo_{\Pb^m_A}(1) = T^*X_A \boxtimes \Oo_{\Pb^m_A}(1) \setminus S \boxtimes \Oo_{\Pb^m_A}(1),$$
$$(T^*X_A,S) \boxtimes \Oo_{\Pb^m_A}(1) = (T^*X_A \boxtimes \Oo_{\Pb^m_A}(1), S_A \boxtimes \Oo_{\Pb^m_A}(1)).$$
\end{enumerate}
\end{convention}

\begin{construction}\label{construction:bracket_nu}
Let $\underline{\nu}=(\nu_1,\dots,\nu_m)$ and $A$ be as in Situation \ref{situation:tuple2}. We construct a morphism $(\underline{\nu})^*\colon \Kb((T^*X_A \setminus S_A)\boxtimes \Oo_{\Pb^m_A}(1)) \to \Kb(U_A)$, such that we have for $i= 1,\dots, m$ a commutative diagram
\[
\xymatrix{
\Kb((T^*X_A \setminus S_A)\boxtimes \Oo_{\Pb^m_A}(1)) \ar[r]^-{(\underline{\nu})^*} \ar[rd]_{L\nu_i^*} & \Kb(U_A) \ar[d] \\
& \Kb((U_i)_A), 
}
\]
of spectra,
and, denoting by $\nu_0 = 0\colon X_A \to T^*X_A$ the zero section, we have a commutative diagram
\[
\xymatrix{
\Kb(T^*X_A\boxtimes \Oo_{\Pb^m_A}(1)) \ar[r] \ar[d]_{L\nu_0^*} & \Kb((T^*X_A \setminus S_A)\boxtimes \Oo_{\Pb^m_A}(1)) \ar[d]^{(\underline{\nu})^*} \\
\Kb(X_A) \ar[r] & \Kb(U_A)
}
\]
of spectra.
\end{construction}

\begin{proof}[Construction]
Let $I \subset  \{1,\dots,m\}$ be a subset, we denote by $(U_I)_A$ the intersection $\bigcap_{i \in I}(U_i)_A$ inside $U_A = \bigcup_{i=1}^m (U_i)_A$. In particular we have $(U_{\emptyset})_A = U_A$. The power set $\Pc(\{1,\dots,m\})$ is ordered by inclusion. The corresponding diagram of schemes 
$$\Pc(\{1,\dots, m\})^{\op} \to \Sch$$
is a pushout diagram. By virtue of the Localisation Theorem for algebraic $K$-theory, we obtain that the resulting diagram of spectra, obtained by applying the functor $\Kb$,
$$\mathcal{U}\colon \Pc(\{1,\dots,m\}) \to \Spectra$$
is a cartesian cubical diagram. By definition, the vertices of this cube are the $K$-theory spectra $\Kb((U_I)_A)$. Let $\Pc'(\{1,\dots,m\})$ denote the set of non-empty subsets. The universal property of cartesian diagrams shows that it suffices to produce a commutative diagram
$$\mathcal{V}\colon \Pc(\{1,\dots,m\}) \to \Spectra,$$
such that $\mathcal{V}|_{\Pc'(\{1,\dots,m\})} = \mathcal{U}|_{\Pc'(\{1,\dots,m\})}$,
and $\mathcal{V}(\emptyset) = \Kb((T^*X_A \setminus S_A)\boxtimes \Oo_{\Pb^m_A}(1))$, and for every $1 \leq i \leq m$ the inclusion $\emptyset \subset \{i\}$ is sent to 
$$\Kb((T^*X_A \setminus S_A)\boxtimes \Oo_{\Pb^m_A}(1)) \to^{L\nu_i^*} \Kb((U_i)_A).$$
We construct the required simplicial map $\mathcal{V}$ inductively. For every non-degenerate $\ell$-simplex 
$$\sigma \colon [\ell] \to \Pc(\{1,\dots,m\})$$ we have to associate an $\ell$-simplex $\mathcal{V}(\sigma)$, consistent with our previous choices for $(\ell -1)$-simplices. The above conditions already dictate a choice for $0$ and $1$-simplices. In order to fill in the remaining gaps we have to produce a commutative diagram of spectra
\[
\xymatrix{
\Kb((T^*X_A \setminus S_A)\boxtimes \Oo_{\Pb^m_A}(1)) \ar[rrrrd] \ar[rd] \ar[rrd] \ar[d] & & & &  \\
\Kb((U_{I_1})_A) \ar[r] & \Kb((U_{I_2})_A) \ar[r] & \Kb((U_{I_3})_A) \ar[r] & \cdots & \Kb(((U_{I_{\ell}})_A)
}
\]
for every chain of subsets $\emptyset = I_0 \subset I_1 \subset \cdots \subset I_{\ell}$. It suffices to produce an $\ell$-simplex
\[
\xymatrix{
\Kb((T^*X_A \setminus S_A)\boxtimes \Oo_{\Pb^m_A}(1)) \ar[rrrrd] \ar[rd] \ar[rrd] \ar[d] & & & &  \\
\Kb((U_{I_{\ell}})_A) \ar[r] & \Kb((U_{I_{\ell}})_A) \ar[r] & \Kb((U_{I_{\ell}})_A) \ar[r] & \cdots & \Kb((U_{I_{\ell}})_A).
}
\]
(all of whose faces are already defined by induction).
Construction \ref{const:homotopy} delivers this $\ell$-simplex to us, as a consequence of the contractibility of the Kan complex $\Psi(i_*)$. 

The property 
$$(L\nu_0^*)|_U \simeq (\underline{\nu})^*(?|_{T^*X_A \boxtimes \Oo_{\Pb^m_A}(1) \setminus S_A \boxtimes \Oo_{\Pb^m_A}(1)})$$
is obtained by applying the construction above to $S = \emptyset$ and the $(m+1)$-tuple of sections $(0, \nu_1,\dots,\nu_m)$.
\end{proof}

\begin{definition}\label{defi:homotopy}
\begin{enumerate}
\item[(a)] Let $(U_i, \nu_i)_{i = 1}^m$ a collection of open subsets $U_i \subset X$ and sections $\nu_i \in \Omega^1_{(U_i)_A/A}$ which satisfies the same condition with respect to $S \subset T^*X$ as in Situation \ref{situation:tuple2}. We refer to such a set $(U_i,\nu_i)_{i=1}^m$ as an \emph{$S$-admissible collection} (relative to $A$). 
\item[(b)] The set of $S$-admissible collections, such that $\bigcup_{i=1}^m U_i = U$ will be denoted by $\Adm_{S,A}(U)$.
\item[(c)] Given two $S$-admissible collections $(U_i, \nu_i)_{i = 1}^m$ and $(U'_i,\nu'_i)_{i = 1}^{m'}$ in $\Adm_{S,A}(U)$, we say 
$$(U_i, \nu_i)_{i = 1}^m < (U'_i,\nu'_i)_{i = 1}^{m'}$$
if the open covering $(U'_i)_{i=1}^{m'}$ refines the open covering $(U_i)_{i=1}^m$, and the union 
$$(U_i, \nu_i)_{i = 1}^m \cup (U'_i,\nu'_i)_{i = 1}^{m'}$$
is still $S$-admissible.
\end{enumerate}
\end{definition}

Applying Construction \ref{construction:bracket_nu} to $(U_i, \nu_i)_{i = 1}^m \cup (U'_i,\nu'_i)_{i = 1}^{m'}$ where $(U_i, \nu_i)_{i = 1}^m < (U'_i,\nu'_i)_{i = 1}^{m'} \in \Adm_{S,A}(U)$, we obtain a homotopy between $(\underline{\nu}_i)^*$ and $(\underline{\nu}')^*$.

\begin{corollary}\label{cor:homotopy}
Let $(U_i, \nu_i)_{i = 1}^m < (U'_i,\nu'_i)_{i = 1}^{m'} \in \Adm_{S,A}(U)$. Then we have a commutative diagram
\[
\xymatrix{
\Kb((T^*X_A \setminus S_A)\boxtimes \Oo_{\Pb^m_A}(1)) \ar[r]^-{(\underline{\nu})^*} \ar[rd]_{(\underline{\nu}')^*} & \Kb(U_A) \ar@{=}[d] \\
& \Kb(U_A) 
}
\]
of spectra,
and, for every triple $(U_i, \nu_i)_{i = 1}^m < (U'_i,\nu'_i)_{i = 1}^{m'} < (U''_i,\nu''_i)_{i = 1}^{m''}   \in \Adm_{S,A}(U)$
the commuting triangles above fit into a commuting tetrahedron
\[
\xymatrix{
& & \Kb(U_A) \ar@{=}[ld] \ar@{=}[ld] \ar@{=}[ldd] \\
\Kb((T^*X_A \setminus S_A)\boxtimes \Oo_{\Pb^m_A}(1)) \ar[r]^-{(\underline{\nu})^*} \ar[urr]^{(\underline{\nu}'')^*} \ar[rd]_{(\underline{\nu}')^*} & \Kb(U_A) \ar@{=}[d] & \\
& \Kb(U_A) &
}
\]
of spectra.
Denoting by $\nu_0 = 0\colon X_A \to T^*X_A$ the zero section, we have a commutative diagram
\[
\xymatrix{
\Kb(T^*X_A\boxtimes \Oo_{\Pb^m_A}(1)) \ar@/_/[rrdd]\ar[r] \ar[d]_{L\nu_0^*} & \Kb((T^*X_A \setminus S_A)\boxtimes \Oo_{\Pb^m_A}(1)) \ar[d]_{(\underline{\nu})^*} \ar[ddr]^{(\underline{\nu}')^*} & \\
\Kb(X_A) \ar[r] \ar[rrd] & \Kb(U_A) \ar@{=}[rd] & \\
& & \Kb(U_A).
}
\]
\end{corollary}

\begin{proof}
The first assertion follows right from the inductive definition of the morphism $(\underline{\nu})$ of Construction \ref{construction:bracket_nu}, applied to $(U_i, \nu_i)_{i = 1}^m \cup (U'_i,\nu'_i)_{i = 1}^{m'}$. The second assertion follows by adding the zero section $(X,\nu_0)$ to $(U_i, \nu_i)_{i = 1}^m \cup (U'_i,\nu'_i)_{i = 1}^{m'}$ and considering the case $S = \emptyset$.
\end{proof}

%\begin{convention}\label{convention}
%For a morphism of spectra $f\colon A \to B$ we formally denote by $B/f(A)$ the cofibre of the morphism $f$.
%\end{convention}

\begin{definition}\label{defi:modulus}
Let $Y\to \Pb^m_k$ be a morphism of schemes, such that the closed immersion
$$Y \times_{\Pb^m_k} \Pb^{m-1}_k \to^i Y$$
has the property that the structure sheaf $i_*\Oo_{Y \times_{\Pb^m_k} \Pb^{m-1}_k}$ is a perfect $\Oo_Y$-module. We denote the cofibre of 
$$\Kb(Y \times_{\Pb^m_k} \Pb^{m-1}_k) \to^{i_*} \Kb(Y)$$
by $\Kb_{\infty}(Y)$. Similarly, for a closed subset $Y_0 \subset Y$ we write $\Kb_{\infty}(Y,Y_0)$ to denote the cofibre of
$\Kb(Y \times_{\Pb^m_k} \Pb^{m-1}_k,Y_0 \times_{\Pb^m_k} \Pb^{m-1}_k) \to^{i_*} \Kb(Y,Y_0)$.
\end{definition}

The second building block we need is a well-defined morphism of spectra 
\begin{equation}\label{eqn:no_A}
\widetilde{Q}_{S,A}\colon\Kb(D_X,S) \to \Kb_{\infty}((T^*X_A,S_A)\boxtimes \Oo_{\Pb^m_A}(1)).
\end{equation}
We refer the reader to Remark \ref{rmk:why_no_A} for an explanation why we don't expect the existence of such a map (fit for our purpose) from $\Kb(D_X \otimes_k A,S_A) \to \Kb((T^*X_A,S_A)\boxtimes \Oo_{\Pb^m_A}(1))$.

\begin{lemma-definition}\label{lemma-defi:tildeQ_S}
There exists a morphism of spectra $\widetilde{Q}_{S,A}$ as in \eqref{eqn:no_A}, such that the diagram of spectra
\[
\xymatrix{
\Kb(D_X,S) \ar@{..>}[r]^-{\exists\widetilde{Q}_{S,A}} \ar[d]_{Q} &  \Kb_{\infty}((T^*X_A,S_A)\boxtimes \Oo_{\Pb^m_A}(1)) \ar[d]^{L\nu_0^*} \\
\Kb(T^*X) \ar[r] & \Kb(T^*X_A)
}
\]
commutes.
\end{lemma-definition}

\begin{proof}
It suffices to construct the morphism $\widetilde{Q}_{S,k}$. Indeed, for a commutative $k$-algebra $A$ we can define 
$$\widetilde{Q}_{S,A} = \widetilde{Q}_{S,k} \otimes_k A,$$
where $\otimes_kA$ denotes the base change morphism
$$\Kb_{\infty}(Y) \to \Kb_{\infty}(Y_A)$$
for a $\Pb^m_k$-scheme $Y$.

Henceforth we assume $A = k$. 
%To simplify notation we introduce the notation
%$$K = \Kb_{\infty}(T^*X\boxtimes \Oo_{\Pb^m}(1),S \boxtimes \Oo_{\Pb^m}(1)).$$
 We have a cocartesian diagram of spectra (see \cite[Theorem 3.1.15]{MR2981817} and Lemma \ref{lemma:bifilt})
\[
\xymatrix{
\Kb(FFD_X,S) \ar[r] \ar[d] & \Kb(FD_X,S) \ar[d] \\
\Kb(FD_X,S) \ar[r] & \Kb(D_X,S).
}
\]
That is, $\Kb(D_X,S)$ is equivalent to the pushout $\Kb(FD_X,S) \sqcup_{\Kb(FFD_X,S)} \Kb(FD_X,S)$. For this reason it suffices to construct a commutative diagram
\[
\xymatrix{
\Kb(FFD_X,S) \ar[r] \ar[d] & \Kb(FD_X,S) \ar[d] \ar@/^10pt/[rdd] & \\
\Kb(FD_X,S) \ar[r] \ar@/_10pt/[rrd] & \Kb_{\infty}((T^*X,S)\boxtimes \Oo_{\Pb^m}(1)) \ar[rd]_-{L\nu_0^*} & \\
& & \Kb(T^*X)
}
\]
and deduce the assertion from the universal property of pushouts. We claim that the existence of such a commutative diagram follows from Proposition \ref{prop:compactification}. Indeed, \emph{loc. cit.} yields a commutative diagram
\[
\xymatrix{
\Kb(FD_X,S) \ar[r] \ar[d] & \Kb((T^*X,S) \boxtimes \Oo_{\Pb^m}(1)) \ar[d] \\
\Kb(D_X,S)  \ar[r] & \Kb(T^*X \times \Ab^m, S \times \Ab^m),
}
\]
which implies that we have a commutative diagram
\[
\xymatrix{
\Kb(FFD_X,S) \ar[r] \ar[d] & \Kb(FD_X) \ar[d] \ar@/^10pt/[rdd] & \\
\Kb(FD_X,S)  \ar[r] \ar@/_10pt/[rrd] & \Kb(T^*X \times \Ab^m, S \times \Ab^m) \ar[rd] & \\
& & \Kb(T^*X) 
}
\]
By virtue of localisation, the spectrum $ \Kb(T^*X \times \Ab^m, S \times \Ab^m)$ is equivalent to the cofibre of
$$\Kb((T^*X,S) \boxtimes \Oo_{\Pb^{m-1}}(1)) \to \Kb((T^*X,S) \boxtimes \Oo_{\Pb^m}(1)).$$
Since the scheme $T^*X \boxtimes \Oo_{\Pb^m}(1)$ is smooth, the cofibre above is isomorphic to 
$$\Kb_{\infty}((T^*X,S) \boxtimes \Oo_{\Pb^m}(1)).$$
This shows what we wanted.
\end{proof}

\begin{rmk}\label{rmk:why_no_A}
We don't expect there to be a morphism 
$$\Kb(D_{X_A/A},S_A) \to  \Kb(T^*X_A\boxtimes \Oo_{\Pb^m_A}(1),S_A \boxtimes \Oo_{\Pb^m_A}(1))$$
fitting into a commutative diagram as in Lemma-Definition \ref{lemma-defi:tildeQ_S}. This is comparable to the \emph{epsilon-niceness} condition of Beilinson--Bloch--Esnault \cite[4.4]{MR1988970} which is needed in order to define a de Rham epsilon line of an $A$-family of irregular flat connections. %We will speculate about a higher-dimensional version of this notion in \ref{subsub:epsilon-nice}.
\end{rmk}

This enables us to finally define the de Rham epsilon factor with respect to the $m$-tuple $(\nu_1,\dots,\nu_m)$ which takes values in $\Kb(A)$, as well as the de Rham epsilon line.

\begin{definition}\label{defi:main_epsilon}
\begin{enumerate}
\item[(a)] The morphism 
$$\widetilde{\phi}_{\underline{\nu},A}\colon \Kb_{\infty}(T^*X_A \boxtimes \Oo_{\Pb^m}(1),S_A \boxtimes \Oo_{\Pb^m}(1)) \to \Kb(X_A,Z_A)$$
is defined to be the induced map of cofibres of the rows of the following commutative square
\[
\xymatrix{
\Kb_{\infty}(T^*X_A \boxtimes \Oo_{\Pb^m}(1)) \ar[r] \ar[d] & \Kb_{\infty}(T^*X_A \boxtimes \Oo_{\Pb^m}(1) \setminus S_A \boxtimes \Oo_{\Pb^m}(1)) \ar[d] \\
\Kb(X_A) \ar[r] & \Kb(X_A \setminus Z_A).
}
\]

\item[(b)] The de Rham epsilon factor is defined to be the composition
$$\Ec_{\underline{\nu},Z} = \widetilde{\phi}_{\underline{\nu}} \circ \widetilde{Q}_{S,A}\colon \Kb(D_X,S) \to \Kb(X_A,Z_A),$$
where $\widetilde{Q}_{S,A}$ is the morphism of Lemma-Definition \ref{lemma-defi:tildeQ_S}.

\item[(c)] If $Z \subset X$ is proper, we define the de Rham epsilon line 
$$\varepsilon_{\underline{\nu}} = \grdet(\tau_{\geq 0}(\pi_*\Ec_{\underline{\nu},Z}))\colon K(D_X,S) \to \grPic(A),$$
where $\tau_{\geq 0}$ denotes the truncation functor from spectra to connective spectra.
\end{enumerate}
\end{definition}

\begin{proposition}\label{prop:main_epsilon}
There is a commutative diagram in the homotopy category of spectra
\[
\xymatrix{
& \Kb(D_X,S) \ar[r]^{\Ec_{\underline{\nu},A}} \ar[ld]_{Q} \ar[d] & \Kb(X_A,Z_A) \ar[d] \\
\Kb(T^*X) \ar[r] & \Kb(T^*X_A) \ar[r]^{L\nu_0^*} & \Kb(X_A).
}
\]
\end{proposition}

\begin{proof}
This follows immediately from Definition \ref{defi:main_epsilon} of the epsilon factor, and the commutative diagram in the homotopy category of spectra
\[
\xymatrix{
\Kb_{\infty}((T^*X_A,S_A) \boxtimes \Oo_{\Pb^{m-1}}(1)) \ar[r] \ar[d]_{\widetilde{\phi}_{\underline{\nu},A}} & \Kb_{\infty}(T^*X_A \boxtimes \Oo_{\Pb^m}(1)) \ar[r] \ar[d] & \Kb_{\infty}(T^*X_A \boxtimes \Oo_{\Pb^m}(1) \setminus S_A \boxtimes \Oo_{\Pb^m}(1)) \ar[d] \\
\Kb(X_A,Z_A) \ar[r] & \Kb(X_A) \ar[r] & \Kb(X_A \setminus Z_A).
}
\]
which exists by virtue of the definition of $\widetilde{\phi}_{\underline{\nu},A}$ as the cofibre of the commuting square on the right (considered as a commutative diagram in the stable $\infty$-category of spectra).
\end{proof}

\begin{theorem}\label{thm:main_product}
Assume that $X$ is proper, then we have a commutative diagram
\[
\xymatrix{
\Kb(D_X,S) \ar[r]^{\Ec_{\underline{\nu},A}} \ar[d]_{R\Gamma_{dR}} & \Kb(X,A) \ar[d]^{R\Gamma} \\
\Kb(k) \ar[r]^{\otimes_k A} & \Kb(A),
}
\]
and consequently we have an isomorphism of functors
$$\grdet(R\Gamma_{dR} \otimes_k A)\simeq \varepsilon_{\underline{\nu},A}\colon K(D_X,S) \to \grPic(A).$$
\end{theorem}

\begin{proof}
There is a commutative diagram 
\[
\xymatrix{
\Kb(D_X) \ar[r]^Q \ar[rrd]_{R\Gamma_{dR}} & \Kb(T^*X) \ar[r]^{L\nu_0^*} & \Kb(X) \ar[d]^{R\Gamma}  \\
& & \Kb(k)
}
\]
(see \cite[Lemma 3.22]{MR2981817}). Base change along $k \to A$ leads to a commutative diagram
\[
\xymatrix{
\Kb(D_X) \ar[r]^Q \ar[rd] & \Kb(T^*X_A) \ar[r]^-{L\nu_0^*} & \Kb(X_A) \ar[r] & \Kb(A)    \\
& \Kb(T^*X) \ar[u] \ar[r] & \Kb(X) \ar[u] \ar[r] & \Kb(k) \ar[u] &
}
\]
in the homotopy category of spectra. Taking into account the commutative diagram of Proposition \ref{prop:main_epsilon}, we conclude the proof.
\end{proof}

Recall the relation $(U_i,\nu_i)_{i=1}^m < (U'_i,\nu'_i)_{i=1}^{m'}$ introduced in Definition \ref{defi:homotopy}. As a consequence of Corollary \ref{cor:homotopy} we obtain the following results:

\begin{proposition}\label{prop:homotopy}
\begin{enumerate}
For every relation $(U_i,\nu_i)_{i=1}^m < (U'_i,\nu'_i)_{i=1}^{m'}$ in $\Adm_{S,A}(U)$ we have a commutative diagram
\item[(a)]
\[
\xymatrix{
\Kb((T^*X_A , S_A)\boxtimes \Oo_{\Pb^m_A}(1)) \ar[r]^-{\widetilde{\phi}_{\underline{\nu},A}} \ar[rd]_-{\widetilde{\phi}_{\underline{\nu}',A}} & \Kb(X_A,Z_A) \ar@{=}[d] \\
& \Kb(X_A,Z_A), 
}
\]
of spectra,
and, 
\item[(b)] for every $(U_i,\nu_i)_{i=1}^m < (U'_i,\nu'_i)_{i=1}^{m'} < (U''_i,\nu_i)_{i=1}^m$ we have a commutative diagram
\[
\xymatrix{
& & \Kb(X_A,Z_A) \ar@{=}[ld] \ar@{=}[ldd]\\
\Kb((T^*X_A , S_A)\boxtimes \Oo_{\Pb^m_A}(1)) \ar[rru]^-{\widetilde{\phi}_{\underline{\nu}'',A}} \ar[r]^-{\widetilde{\phi}_{\underline{\nu},A}} \ar[rd]_-{\widetilde{\phi}_{\underline{\nu}',A}} & \Kb(X_A,Z_A) \ar@{=}[d] & \\
& \Kb(X_A,Z_A), & 
}
\]
of spectra.
\item[(c)] Denoting by $\nu_0 = 0\colon X_A \to T^*X_A$ the zero section, the commuting diagram of (a) extends to a commutative diagram
\[
\xymatrix{
\Kb(T^*X_A\boxtimes \Oo_{\Pb^m_A}(1)) \ar[rdd] \ar[d]_{L\nu_0^*} & \Kb((T^*X_A , S_A)\boxtimes \Oo_{\Pb^m_A}(1)) \ar[l] \ar[d]_-{\widetilde{\phi}_{\underline{\nu},A}} \ar[ddr]^-{\widetilde{\phi}_{\underline{\nu}',A}} & \\
\Kb(X_A) & \Kb(X_A,Z_A)   \ar@{=}[rd] \ar[l] & \\
& \Kb(X_A) \ar[lu] & \Kb(X_A,Z_A).\ar[l]
}
\]
\item[(d)] If we assume in addition properness of $X$ then we have a commutative diagram, then we have a commutative diagram
\[
\xymatrix{
& & \Kb(X_A,Z_A) \ar[ldd]^{R\Gamma} \\
\Kb(D_X,S) \ar[rru]^-{\Ec_{\underline{\nu}',A}} \ar[r]_-{\Ec_{\underline{\nu},A}} \ar[d]_{R\Gamma_{dR}} & \Kb(X,A) \ar[d]^{R\Gamma} \ar@{=}[ru] \\
\Kb(k) \ar[r]^{\otimes_k A} & \Kb(A),
}
\]
\end{enumerate}
in particular we obtain a commutative diagram of natural isomorphism 
\[
\xymatrix{
& & \grPic(A) \\
K(D_X,S) \ar[r]^-{\varepsilon_{\underline{\nu},A}} \ar[rru]^{\varepsilon_{\underline{\nu}',A}} \ar[rd]_{\grdet(R\Gamma_{dR})} & \grPic(A) \ar@{=}[ru] \ar@{=}[d] & \\
& \grPic(A). \ar@{=}[uur] &
}
\]
\end{proposition}

We infer from these statements the existence of the epsilon connection. This is the content of the subsequent section. Before turning to the construction of this canonical connection we explain how to deduce main theorem stated in the introduction from the results of this subsection.

\begin{proof}[Proof of Theorem \ref{thm:mainmain}]
The theorem follows from \ref{prop:main_epsilon} and the remark that $K(D_X,\underline{\nu}) \simeq \colim K(D_X,S)$ where $S \subset T^*X$ ranges over all subsets, such that $\underline{\nu}$ is admissible.
\end{proof}

%\begin{question}
%In the one-dimensional theory of de Rham epsilon factors a special role is played by Deligne's good pairs of lattices. Is there a higher-dimensional analogue of this notion? If yes, can it be used to express higher-dimensional epsilon factors as studied by Patel and by the author?
%\end{question}

\section{The epsilon connection}\label{epsilon}

In this short section we explain how $\Pb^1$-homotopies can be used to produce a connection which arises in the shape of a crystal.

For a presheaf $F\colon \Rings_k \to \Sets$, the presheaf $F^{dR}$ is given by the functor sending $A \mapsto F(A^{\red})$. One defines $\grPic(F)$ to be the colimit 
$$\grPic(F) = \varinjlim_{A \to F} \grPic(A).$$
There is a well-known correspondence between $\grPic(F^{dR})$ and crystals of graded lines on $F$. The map $F \to \Spec k$, and the fact $(\Spec k)^{dR} = \Spec k$ induces a natural functor $\grPic(k) \to \grPic(F^{dR})$. Crystals of graded lines in the essential image of this functor will be called \emph{constant}.
\begin{definition}
Let $X$, $Z$, $S$, $A$, $U$, $\{U_i\}_{i=1,\dots, m}$ be as in Situation \ref{situation:tuple2}. We denote by $\bfOmega\colon \Rings_k \to \Sets$ the functor sending a commutative $k$-algebra $A$ to the set of tuples $\underline{\nu}$ satisfying the condition of Situation \ref{situation:tuple2}.
\end{definition}

\begin{situation}\label{situation:connection}
Let $X$, $Z$, $S$, $A$, $U$, $\{U_i\}_{i=1,\dots, m}$ be as in Situation \ref{situation:tuple2}. For $j=1,2,3$, we denote by $\underline{\nu}^j = (\nu^j_1,\dots,\nu^j_m)$ an $m$-tuple of $1$-forms $\nu_i^j \in \Omega^1(U_i/A)$ satisfying the condition of Situation \ref{situation:tuple2} for each $j$. Furthermore, we assume that $\nu_i^1|_{A^{\red}} = \nu_i^2|_{A^{\red}} = \nu_i^3|_{A^{\red}}$.
\end{situation}

\begin{corollary}\label{cor:main_epsilon_connection}
Assume that $Z$ is proper over $k$. For $\underline{\nu}^1$ and $\underline{\nu}^1$ as in Situation \ref{situation:connection} there exists an isomorphism
$$\cri_{12}\colon \varepsilon_{\underline{\nu}^1,A} \simeq \varepsilon_{\underline{\nu}^2,A}.$$
Furthermore, for $\underline{\nu}^1$, $\underline{\nu}^2$ and $\underline{\nu}^3$ as in Situation \ref{situation:connection} we have a commutative diagram 
\[
\xymatrix{
\varepsilon_{\underline{\nu}^1,A} \ar[r]^{\cri_{12}} \ar[rd]_{\cri_{13}} & \varepsilon_{\underline{\nu}^2,A} \ar[d]^{\cri_{23}} \\
& \varepsilon_{\underline{\nu}^3,A}
}
\]
of isomorphism. Furthermore, if $X$ is proper, then these isomorphisms belong to a commutative diagram
\[
\xymatrix{
\varepsilon_{\underline{\nu}^1,A} \ar[r]^{\cri_{12}} \ar[d] \ar[rd]_{\cri_{13}} & \varepsilon_{\underline{\nu}^2,A} \ar[d]^{\cri_{23}} \\
\grdet(R\Gamma_{dR}) \otimes_k A \ar[r] & \varepsilon_{\underline{\nu}^3,A}.
}
\]
\end{corollary}

\begin{proof}
We claim that for $\underline{\nu}^1$ and $\underline{\nu}^2$ as in Situation \ref{situation:connection} we have 
$$(U_i,\nu_i^1)_{i=1}^m < (U_i,\nu_i^2)_{i=1}^m$$
in $\Adm_{S,A}(U)$. The assertion the follows from Proposition \ref{prop:homotopy}.
In order to verify the claim it suffices to check that the union $(U_i,\nu_i^1)_{i=1}^m \cup (U_i,\nu_i^2)_{i=1}^m$ is still $S$-admissible. Let  $\lambda_1,\cdots,\lambda_{2m}$, such that we have a relation
$$\sum_{i=1}^m\lambda_i \nu_i^1 + \sum_{i=1}^m \lambda_{m+i} \nu_i^2 \cap S \neq \emptyset.$$
We apply the homomorphism $A \to A^{\red}$ (recall $\nu_i^1|_{A^{\red}} = \nu_i^2|_{A^{\red}}$) and obtain
$$\sum_{i=1}^m (\lambda_i + \lambda_{m+i}) \nu_i^1|_{A^{\red}} \cap S \neq \emptyset.$$
This contradicts the assumption of $S$-admissibility of $(U_i,\nu_i)_{i=1}^m$.

Consequently we have an isomorphism of graded lines $\varepsilon_{\underline{\nu}_1,A} \simeq \varepsilon_{\underline{\nu}_2,A}$, satisfying the crystalline cocycle condition (and hence $\varepsilon^{\cri}$ is a well-defined crystal on $\bfOmega$). For $X$ a proper scheme, we have an isomorphism of the crystal
$\varepsilon^{\cri}$ with the constant crystal $\grdet(R\Gamma_{dR}^*)_{\bfOmega}$ on $\bfOmega$ as a consequence of Proposition \ref{prop:homotopy}. 
\end{proof}

In light of our remarks on crystals, the theorem above yields the following corollary.

\begin{corollary}\label{cor:main_epsilon_connection2}
The morphism $\varepsilon\colon K(D_X,S) \to \grPic(\bfOmega_{X,S})$ factors through $\grPic(\bfOmega_{X,S}^{dR})$. That is, defines a crystal of graded lines on $\bfOmega_{X,S}$. We denote the induced map by $\varepsilon^{\cri}\colon K(D_X,S) \to \grPic(\bfOmega_{X,S})$.
Furthermore, if $X$ is proper, then $\varepsilon^{\cri} \simeq \grdet(R\Gamma_{dR}) \otimes_k A$ with the constant crystal structure.
\end{corollary}

We don't know if for $X$ a curve, our epsilon connection is the same as the epsilon connection constructed by Beilinson--Bloch--Esnault in \cite{MR1988970}.

\appendix

\section{Notation}

\begin{tabular}{p{1.75cm}p{12cm}}
$k$ & field of characteristic $0$\\
$A$ & commutative $k$-algebra (with unit)\\
$X$ & smooth $k$-scheme, also the case of a trait, that is, $\Spec k'[[t]]$ is allowed, where $k'/k$ is a finite field extension\\
$U$ & open subset of $X$\\
$Z$ & closed complement of $U$ (reduced) \\
$\{U_i\}_{i=1}^m$ & open covering of $U$ \\
$X_A$, $U_A$ & base change $X \times_k A$, respectively $U \times _k A$\\
$\nu$ & algebraic $1$-form defined on $U$, or a section of $\Omega^1(U_A/A)$ \\
$\underline{\nu} = (\nu_i)_{i=1}^m$ & sections of $\Omega^1_{U_A/A}((U_i)_A)$ \\ 
$\grPic(A)$ & Picard groupoid of graded invertible $A$-modules (or, graded lines) \\
$\grdet$ & graded determinant of a vector space, or projective $A$-module: $\grdet V = (\rk V, \det V)$ \\
$\varepsilon^{BBE}_{\nu}$ & graded epsilon line as defined in \cite{MR1988970}, see also \ref{BBE}\\
$K$ & connective algebraic $K$-theory spectrum of an exact category, a stable $\infty$-category, or a scheme \\
$\Kb$ & non-connective algebraic $K$-theory spectrum of an exact category, a stable $\infty$-category, or a scheme \\
$\Kb_{\infty}$ & see Definition \ref{defi:modulus} \\
$K(D_X)$ & algebraic $K$-theory of locally finitely presented right $D$-modules \\
$\Ec^{P}_{\nu}$ & Patel's epsilon factor taking values in $K(X,D)$, as defined in \cite{MR2981817}, see also \ref{Patel}\\
$\varepsilon^{P}_{\nu}$ & graded epsilon line as defined in \cite{MR2981817}, see also \ref{Patel}\\
$\phi_{\nu}$ & see \ref{Patel} \\
$\Ec_{\underline{\nu}}$ & epsilon factor taking values in $\Kb(X_A,Z_A)$, see Definition \ref{defi:main_epsilon} \\
$\varepsilon_{\underline{\nu}}$ & graded epsilon line, see Definition \ref{defi:main_epsilon}(c) \\
$\tilde{\phi}_{\underline{\nu},A}$ & see Definition \ref{defi:main_epsilon}(a) \\ 
$(\underline{\nu}^*)$ & see Construction \ref{construction:bracket_nu} \\
$\Grpd$ & $\infty$-category of small $\infty$-groupoids, or Kan complexes, or spaces\\
$\Gpd$ & $2$-category of groupoids \\
$\simeq$ & denotes an equivalence of two objects in an $\infty$-category, including the special case of an equivalence of two small $\infty$-categories, or a homotopy between paths, etc. \\
\end{tabular}

\bibliographystyle{amsalpha}
\bibliography{epsilon.bib}

\bigskip
\noindent E-mail: \url{m.groechenig@fu-berlin.de}\\
Address: FU Berlin, Arnimallee 3, 14195 Berlin
\end{document}